\documentclass[onecolumn,smallcondensed,runningheads,envcountsect]{svjour3}
\usepackage{amsmath}
\usepackage{amssymb}
\usepackage{romannum}
\usepackage[english]{babel}
\usepackage{mathptmx}
\AtBeginDocument{\pagenumbering{arabic}}
\usepackage{hyperref}
\allowdisplaybreaks
\smartqed

\begin{document}

\numberwithin{equation}{section}

\newtheorem{thm}{Theorem}[section]

\spnewtheorem{rem}[thm]{Remark}{\bf}{}

\newtheorem{lem}[thm]{Lemma}

\newtheorem{cor}[thm]{Corollary}

\spnewtheorem{ex}[thm]{Example}{\bf}{}

\newtheorem{pro}[thm]{Proposition}

\spnewtheorem{defn}[thm]{Definition}{\bf}{}

\title{Characterizations of weak R-duality and its application to Gabor frames \thanks{Himanshi Bansal (Inspire Code: IF190609) acknowledges the financial support of Department of Science and Technology. S. Arati acknowledges the financial support of National Board for Higher Mathematics, Department of Atomic Energy (Government of India).}}
\titlerunning{Characterizations of weak R-duality}
\author{Himanshi Bansal \and P. Devaraj* \and S. Arati }

\institute{*Corresponding author \at School of Mathematics, Indian Institute of Science Education and Research Thiruvananthapuram, Vithura, Thiruvananthapuram-695551.\\
\email{hbansal20@iisertvm.ac.in, devarajp@iisertvm.ac.in, aratishashi@iisertvm.ac.in}}

\date{Received: date / Accepted: date}
\maketitle

\begin{abstract}
Weak R-duals, a generalization of R-duals, were recently introduced; for which duality relations were established. In this paper, we consider the problem of characterizing a given frame sequence to be a weak R-dual of a given frame. Further, we apply these characterization results to the Gabor frame setting and prove that the weak R-duality of the adjoint Gabor system leads to the R-duality of the same, thereby indicating an approach to answer the famous problem of the adjoint Gabor system being an R-dual of a given Gabor frame.


\end{abstract}

\section{\textbf{Introduction}}
The abstract notion of frames for Hilbert spaces was given by Duffin and Schaeffer \cite{duffin1952class} in 1952 while studying problems in nonharmonic Fourier series, although Gabor in 1946 had initiated this approach while decomposing signals in terms of modulates and translates of functions. After the historical work of Daubechies, Grossmann and Meyer \cite{daubechies1986painless} in 1986, frames gained more significance. Nowadays, frames are being extensively applied in numerous areas and among frames, Gabor frames are one of the most widely used ones. In fact, it has an edge over the Fourier transform in certain applications, such as speech and music signal processing, which involve time-frequency localization. One of the prominent and fundamental results in Gabor analysis is the Duality principle, due to Janssen \cite{janssen1994duality}, Daubechies, Landau and Landau \cite{daubechies1994gabor}, and Ron and Shen \cite{ron1997weyl}, which characterizes Gabor frames in $L^2(\mathbb{R})$ with lattice points $\{(mb,na):m,n \in \mathbb{Z}\},\ a,b>0$, in terms of Gabor systems with lattice points $\left\{\left(\frac{m}{a},\frac{n}{b}\right):m,n \in \mathbb{Z}\right\}$. More explicitly, denoting $E_x$ and $\mathcal{T}_x$, $x\in\mathbb{R}$ to be the modulation and translation operators on $L^2(\mathbb{R})$, given by $(E_x(f))(t):=e^{2\pi i xt}f(t)$ and $(\mathcal{T}_x(f))(t):=f(t-x)$ respectively, the Duality principle states that the Gabor system $\{E_{mb}\mathcal{T}_{na}f\}_{m,n \in \mathbb{Z}},\ a,b>0$, is a frame for $L^2(\mathbb{R})$  with bounds $A, B>0$ if and only if the adjoint Gabor system $\left\{\frac{1}{\sqrt{ab}}E_{\frac{m}{a}}\mathcal{T}_{\frac{n}{b}}f\right\}_{m,n \in \mathbb{Z}}$ is a Riesz sequence in $L^2(\mathbb{R})$ with the same bounds $A, B$.
\par
Casazza, Kutyniok and Lammers \cite{casazza2004duality} in 2004 considered the question of extending the above important principle to the abstract framework. Let $H$ denote an infinite dimensional separable Hilbert space and $I$ denote a countable index set. Casazza, et al. introduced the concept of an R-dual of a sequence in $H$, the definition of which is given below.
\begin{defn}\cite{casazza2004duality} Let $u=\{u_i\}_{i \in I}$, $v=\{v_i\}_{i \in I}$ be orthonormal bases for $H$ and $f=\{f_i\}_{i \in I}$ be a sequence in $H$ such that
$\sum\limits_{i \in I}|\langle f_i,u_j \rangle|^2 < \infty, \ \forall \ j \in I$.
Then, the sequence $w=\{w_j\}_{j \in I}$ given by
\begin{equation*}
w_j=\sum_{i \in I}\big\langle f_i,u_j \big\rangle v_i ,\ j \in I,
\end{equation*}
is said to be an R-dual of $f$ with respect to the orthonormal bases $u$ and $v$.
\end{defn}
Furthermore, they showed that the relations between a sequence and its R-dual bear a close resemblance to those between the Gabor systems $\{E_{mb}\mathcal{T}_{na}f\}_{m,n \in \mathbb{Z}}$ and $\left\{\frac{1}{\sqrt{ab}}E_{\frac{m}{a}}\mathcal{T}_{\frac{n}{b}}f\right\}_{m,n \in \mathbb{Z}}$. In this context, they posed the natural question towards generalizing the Duality principle for Gabor systems: `Given a Gabor frame $\{E_{mb}\mathcal{T}_{na}f\}$, is the adjoint system $\left\{\frac{1}{\sqrt{ab}}E_{\frac{m}{a}}\mathcal{T}_{\frac{n}{b}}f\right\}$ an R-dual of the Gabor frame?'. They answered this question affirmatively only in the case of tight Gabor frames and Gabor Riesz bases. However, Stoeva and Christensen in \cite{stoeva2015r,stoeva2016various} introduced the notion of R-duals of type \Romannum{2}, \Romannum{3} and \Romannum{4}, and showed that a sub-class of the R-duals of type \Romannum{3} generalizes the Duality principle and at the same time retains all the appealing properties of R-duals. Nevertheless, the original question of whether an R-dual will or will not generalize the Duality principle for Gabor systems is an interesting one that remains unanswered.
\par
On the other hand, the study on R-duals and its variants gained momentum over a period of time out of independent interest. Christensen, Kim and Kim \cite{christensen2011duality} in 2011 gave a characterization for a Riesz sequence to be an R-dual of a given frame, which was further modified by Chuang and Zhao \cite{chuang2015equivalent} in 2015. Several generalizations of the concept of R-duals have been developed over the past few years. For instance, generalized R-duals or g-R-duals for the so-called g-frames, which involve bounded linear operators, were studied in \cite{enayati2017duality,takhteh2017r,li2019characterizing}. Similar concepts, namely, Hilbert-Schmidt R-duality in the Hilbert-Schmidt frame theory and $F_a$-R-duality in $F_a$-frame theory were looked into by Dong and Li in \cite{dong2021duality} and Li and Hussain in \cite{li2021duality} respectively. Recently, in 2020, Li and Dong \cite{li2020class} defined another generalization of the R-dual called the weak R-dual, which makes use of Parseval frames instead of orthonormal bases. More precisely, the definition of a weak R-dual is as follows and we make use of the following notation for the same. For sequences $g=\{g_i\}_{i \in I}$ and $h=\{h_i\}_{i \in I}$ in $H$, we denote the matrix $\big(\big\langle g_i,h_j \big\rangle\big)_{i,j \in I}$ by $G(g,h)$, its transpose by $G(g,h)^t$ and the infinite identity matrix by $\mathcal{I}$.
\begin{defn} \cite{li2020class} Let $u=\{u_i\}_{i \in I}$, $v=\{v_i\}_{i \in I}$ be Parseval frames for $H$ and $f=\{f_i\}_{i \in I}$ be a sequence in $H$ such that
\begin{equation*}
\sum_{i \in I}|\langle f_i,u_j \rangle|^2 < \infty, \ \forall\ j \in I
\hspace{1em}and \hspace{1em}
(G(v,v)^t-\mathcal{I})G(f,u)=\bold{0}.
\end{equation*}
Then, the sequence $w=\{w_j\}_{j \in I}$ given by
\begin{equation*}
w_j=\sum_{i \in I}\big\langle f_i,u_j \big\rangle v_i ,\ j \in I,
\end{equation*}
is called a weak R-dual of $f$ with respect to the Parseval frames $u$ and $v$.
\end{defn}
Hence, an R-dual is always a weak R-dual but not vice-versa. Further, Li and Dong established that the weak R-dual of a frame (Riesz basis) is a frame sequence (frame) with the same bounds, which is quite different from the relation that exists in the framework of R-duals. The same authors in \cite{li2021g-duality} also considered weak g-R-duals and presented the duality relations even in this setting.
\par
In this paper, we consider the notion of weak R-duality of frames in an infinite dimensional separable Hilbert space $H$. We provide several characterizations for a given frame sequence to be a weak R-dual of a given frame. A weak R-dual in the real sense is one  wherein either or both of $u$ and $v$ are Parseval frames, but not orthonormal bases. Different necessary and sufficient conditions are proved for the various scenarios. Many of our results are also elucidated using examples. Further, conditions under which a weak R-dual produces an R-dual have been obtained. We then apply the theory developed so far to the context of Gabor frames and conclude that for a given Gabor frame, obtaining the adjoint Gabor system as a weak R-dual paves way to realizing it as an R-dual as well. It is also shown that the adjoint Gabor system does not exist as a weak R-dual in the true sense for Gabor Riesz basis. However, it is proved that it does, in the case of tight Gabor frames.
\par
We shall now provide the necessary background on frames needed for this paper. Let $I$ be a countable index set. A sequence of vectors $f=\{f_i\}_{i \in I}$ in $H$ is said to be a frame for $H$ if there exist $0<A\leq B<\infty$ such that
\begin{equation*}
A ||x||^2\leq \sum_{i \in I} \left|\langle x,f_i \rangle\right|^2\leq B ||x||^2,\ \forall\ x \in H.
\end{equation*}
The constants $A$ and $B$ are called the frame bounds of $f$. The sequence $f$ is said to be a frame sequence in $H$ if it is a frame for $\overline{span}\{f_i\}_{i \in I}$. If, in the above definition of a frame, only the right hand side inequality holds, then $f$ is said to be a Bessel sequence. If $A=B$ in the above inequality, then $f$ is said to be a tight frame and in particular, if $A=B=1$, then it is called a Parseval frame.
For a frame $f$, the bounded linear operator $T_f:\ell^2(I)\longrightarrow H$, given by $T_f\left(\{c_i\}\right)=\sum\limits_{i \in I} c_if_i$, is called the synthesis operator and its adjoint operator $T_f^*:H \longrightarrow \ell^2(I)$, given by $T_f^*(x)=\{\langle x,f_i\rangle\}_{i \in I}$, is called the analysis operator. The operator $S_f:H\longrightarrow H$ given by $S_f(x)=T_fT_f^*(x)=\sum\limits_{i \in I} \big \langle x,f_i \big \rangle f_i$ is called the frame operator of $f$. It is invertible and the sequence $S_f^{-1}f=\{S_f^{-1}f_i\}_{i \in I}$ is also a frame for $H$, which is called the canonical dual frame of $f$.
If $A ||c||^2\leq \left\|\sum\limits_{i \in I} c_if_i \right\|^2\leq B ||c||^2$, for all finite scalar sequences $c=\{c_i\}_{i \in I}$ and for some scalars $0<A\leq B<\infty$, then $\{f_i\}_{i \in I}$ is called a Riesz sequence in $H$. In addition, if $\overline{span}\{f_i\}_{i \in I}=H$, then $\{f_i\}_{i \in I}$ is called a Riesz basis for $H$. Without loss of generality, we shall assume that $I=\mathbb{N}$ in our proofs, in order to avoid cumbersome notation.
\par
We organize our paper as follows. The various necessary as well as sufficient conditions for weak R-duality are discussed in Section 2. These conditions make use of a particular sequence $y$, which we analyse in Section 3. In particular, we look into its completeness and frame properties. The connection between weak R-duality and R-duality as well as the study on the weak R-duality of Gabor frames are the contents of the last section.

\section{Characterization of weak R-duals}
In this section, we shall provide a few results on the characterization of weak R-duality. However, first we would like to emphasize that weak R-duality of frames does not reduce to R-duality in general, even though the definition involves Parseval frames satisfying a certain condition. In the following proposition, for a given frame $f$ and a sequence $u$ in $H$, we shall construct a Parseval frame $v$, which is not an orthonormal basis, satisfying the condition $(G(v,v)^t-\mathcal{I})G(f,u)=\bold{0}$, when $f$ is not a Riesz basis.
\begin{pro}\label{pro1}
Let $f=\{f_i\}_{i \in I}$ be a frame for $H$ but not a Riesz basis and $u=\{u_i\}_{i \in I}$ be any sequence in $H$. Then, there exists a Parseval frame $v=\{v_i\}_{i \in I}$, which is not an orthonormal basis, for $H$ satisfying $(G(v,v)^t-\mathcal{I})G(f,u)=\bold{0}$.
\end{pro}
\begin{proof}
Define an operator $\widetilde{L}:H \longrightarrow H$ by $\widetilde{L}\left(\sum\limits_{i \in I}c_ih_i\right)=\sum\limits_{i \in I}\overline{c_i}h_i$, where $\{h_i\}_{i \in I}$ is an orthonormal basis of $H$ and $\{c_i\}_{i \in I} \in \ell^2(I)$. It can be easily seen that $\widetilde{L}$ is a conjugate-linear and surjective isometry. For $x, z \in H$, we get
\begin{align*}
\langle \widetilde{L}x,z \rangle&=\left\langle \widetilde{L}\left(\sum_{i \in I} \langle x,h_i \rangle h_i\right),z \right\rangle
=\left\langle \sum_{i \in I} \langle h_i,x \rangle h_i,z \right\rangle
=\sum_{i \in I} \langle h_i,x \rangle \langle h_i,z \rangle \\
&=\left\langle \sum_{i \in I} \langle h_i,z \rangle h_i,x \right\rangle
=\left\langle \widetilde{L}\left(\sum_{i \in I} \langle z,h_i \rangle h_i\right),x \right\rangle
=\langle \widetilde{L}z,x \rangle.
\end{align*}
Making use of the above relation, we have
\begin{align*}
\sum_{k \in I}\big\langle \widetilde{L}(S_f^{-1/2}f_k),\widetilde{L}(S_f^{-1/2}f_i) \big\rangle f_k&=\sum_{k \in I}\big\langle \widetilde{L}^2(S_f^{-1/2}f_i),S_f^{-1/2}f_k \big\rangle f_k\\
&=\sum_{k \in I}\big\langle S_f^{-1/2}f_i,S_f^{-1/2}f_k \big\rangle f_k\\
&=\sum_{k \in I}\big\langle f_i,S_f^{-1}f_k \big\rangle f_k\\
&=f_i.
\end{align*}
By taking the inner product with $u_j$ on both sides of the above equality, we obtain $(G(v,v)^t-\mathcal{I})G(f,u)=\bold{0}$, with $v_i=\widetilde{L}(S_f^{-1/2}f_i), i \in I$. Further, as $\{S_f^{-1/2}f_i\}_{i \in I}$ is a Parseval frame for $H$, we have
\begin{equation*}
\sum\limits_{i \in I}|\langle x,\widetilde{L}(S_f^{-1/2}f_i) \rangle|^2=\sum\limits_{i \in I}|\langle S_f^{-1/2}f_i,\widetilde{L}(x)\rangle|^2=\|\widetilde{L}(x)\|^2=\|x\|^2,\ \forall\ x \in H.
\end{equation*}
This proves that $\{\widetilde{L}(S_f^{-1/2}f_i)\}_{i \in I}$ is a Parseval frame for $H$. Moreover,
\begin{equation*}
\langle \widetilde{L}(S_f^{-1/2}f_i),\widetilde{L}(S_f^{-1/2}f_j)\rangle=\langle \widetilde{L}^2(S_f^{-1/2}f_j),S_f^{-1/2}f_i\rangle=\langle S_f^{-1/2}f_j,S_f^{-1/2}f_i\rangle=\langle S_f^{-1}f_j,f_i\rangle\neq \delta_{ji},
\end{equation*}
for some $i,j,$  as $f$ is not a Riesz basis. This, in turn, shows that $\{\widetilde{L}(S_f^{-1/2}f_i)\}_{i \in I}$ is not orthonormal.
\qed\end{proof}

\begin{rem}\label{rem}
If $f$ is indeed a Riesz basis and $u$ is a Parseval frame for $H$, then any Parseval frame $v$ satisfying $(G(v,v)^t-\mathcal{I})G(f,u)=\bold{0}$ will in fact be an orthonormal basis, for, if $v$ satisfies the above condition, then $\sum\limits_{k \in I}\langle v_k,v_i\rangle f_k-f_i$ is orthogonal to all $u_j$'s and hence to all $z \in H$. In particular, $\sum\limits_{k \in I}\langle v_k,v_i\rangle \langle f_k,S_f^{-1}f_j \rangle=\langle f_i,S_f^{-1}f_j\rangle$, which gives $\langle v_j,v_i\rangle=\delta_{ij}$, by the biorthogonality of $f$ and $S_f^{-1}f$. Moreover, if the weak R-dual $w$ of $f$ with respect to $u$ and $v$ is a Riesz basis as well, then it can be seen in \cite{li2020class} that $u$ is an orthonormal basis for $H$. Under these assumptions, the weak R-dual $w$ turns out to be an R-dual itself.
\end{rem}

We shall now prove a few characterization results for weak R-duality, which is the central focus of our paper. The following theorem gives a characterization in terms of the form of the Parseval frame $v$ involved in the definition of a weak R-dual. It in fact plays a significant role in the subsequent analysis.

\begin{thm}\label{thm1}
Let $w=\{w_j\}_{j \in I}$ be a frame sequence, $f=\{f_i\}_{i \in I}$ be a frame, $u=\{u_i\}_{i \in I}$ and $v=\{v_i\}_{i \in I}$ be Parseval frames for $H$. Then, $w$ is a weak R-dual of $f$ with respect to $u$ and $v$ if and only if the following conditions hold:
\begin{align}
&(i)\ (G(\widetilde{w},w)^t-\mathcal{I})G(u,f)=\bold{0}\hspace{21.7em}\label{cond1} \\
&(ii)\ v_i=y_i+x_i,\nonumber
\end{align}
where $\widetilde{w}=\{\widetilde{w}_j\}_{j \in I}$ is the canonical dual of $w$, $y=\{y_i\}_{i \in I}$ is the sequence in $\overline{span}\{w_j\}_{j \in I}$ defined by
\begin{equation}\label{def of ni}
y_i:=\sum_{k \in I}\big\langle u_k,f_i \big\rangle \widetilde{w}_k ,\ i \in I
\end{equation} and $x_i \in (span\{w_j\})^{\perp}$.
\end{thm}
\begin{proof}
Let us assume that $w$ is a weak R-dual of $f$ with respect to $u$ and $v$. In other words, we have $w_j=\sum\limits_{i \in I}\big\langle f_i,u_j \big\rangle v_i ,\ j \in I$ and $\sum\limits_{i \in I}\langle v_i,v_k \rangle \langle f_i,u_j \rangle=\langle f_k,u_j \rangle,\ \forall\ k,j \in I$. From these two expressions, we get $\langle f_k,u_j \rangle=\langle w_j,v_k \rangle$. Further, for $i,j \in I$ we obtain,
\begin{align*}
\sum\limits_{k \in I}\big\langle \widetilde{w}_k,w_j \big\rangle \big\langle u_k,f_i \big\rangle&=\sum\limits_{k \in I}\big\langle \widetilde{w}_k,w_j \big\rangle \langle v_i,w_k \rangle
=\left\langle v_i,\sum\limits_{k \in I}\big\langle w_j,\widetilde{w}_k \big\rangle w_k \right\rangle
=\langle v_i,w_j \rangle
=\langle u_j,f_i \rangle,
\end{align*} which gives that $(G(\widetilde{w},w)^t-\mathcal{I})G(u,f)=\bold{0}$. Also, using the expression of $y_i$, it can be rewritten as $\langle y_i,w_j \rangle=\langle u_j,f_i \rangle$ and hence $\langle w_j,y_i \rangle=\langle w_j,v_i\rangle,\ \forall\ j \in I$. Taking $x_i:=v_i-y_i$, we then have $x_i \in (span\{w_j\})^{\perp}$.
\par
Conversely, suppose that (\ref{cond1}) holds and $v_i=y_i+x_i$, where $y_i$ is given by (\ref{def of ni}) and $x_i \in (span\{w_j\})^{\perp}$. From (\ref{cond1}), it follows that $\sum\limits_{k \in I}\big\langle \widetilde{w}_k,w_i \big\rangle \big\langle u_k,f_j \big\rangle=\big\langle u_i,f_j \big\rangle$ and from (\ref{def of ni}), we get $\sum\limits_{k \in I}\big\langle u_k,f_j \big\rangle \big\langle \widetilde{w}_k,w_i \big\rangle=\big\langle y_j,w_i \big\rangle.$ So,
\begin{equation}\label{eq6}
\big\langle f_i,u_j \big\rangle=\big\langle w_j,y_i \big\rangle.
\end{equation}
Using (\ref{eq6}), for $j \in I$, we have
\begin{align*}
\sum_{i \in I}\langle f_i,u_j\rangle v_i&=\sum_{i \in I}\langle w_j,y_i \rangle (y_i+x_i)
=\sum_{i \in I}\langle w_j,y_i \rangle y_i+\sum_{i \in I}\langle w_j,y_i \rangle x_i
=w_j+\sum_{i \in I}\langle w_j,y_i \rangle x_i,
\end{align*} as $y$, being the orthogonal projection of the Parseval frame $v$ onto $\overline{span}\{w_j\}_{j \in I}$, is a Parseval frame for $\overline{span}\{w_j\}_{j \in I}$. Now, we will prove that $\sum\limits_{i \in I}\langle w_j,y_i \rangle x_i=0$. For $z=z_1+z_2 \in H$, with $z_1 \in \overline{span}\{w_j\}_{j \in I}$ and $z_2 \in (span\{w_j\})^{\perp}$, consider
\begin{align*}
\|z\|^2&=\sum_{i \in I}|\langle z,v_i \rangle|^2
=\sum_{i \in I}|\langle z_1+z_2,y_i+x_i \rangle|^2
=\sum_{i \in I}|\langle z_1,y_i \rangle+ \langle z_2,x_i \rangle|^2\\
&=\sum_{i \in I}|\langle z_1,y_i \rangle|^2+\sum_{i \in I}|\langle z_2,x_i \rangle|^2+\sum_{i \in I}\langle z_1,y_i \rangle \overline{\langle z_2,x_i \rangle}+\sum_{i \in I}\overline{\langle z_1,y_i \rangle} \langle z_2,x_i \rangle\\
&=\|z_1\|^2+\|z_2\|^2+2Re\left(\sum_{i \in I}\langle z_1,y_i \rangle \overline{\langle z_2,x_i \rangle}\right)\\
&=\|z\|^2+2Re\left(\sum_{i \in I}\langle z_1,y_i \rangle \overline{\langle z_2,x_i \rangle}\right),
\end{align*} which implies that $Re\left(\sum\limits_{i \in I}\langle z_1,y_i \rangle \overline{\langle z_2,x_i \rangle}\right)=0$. Similarly, we get $Im\left(\sum\limits_{i \in I}\langle z_1,y_i \rangle \overline{\langle z_2,x_i \rangle}\right)=0$, by taking $z=iz_1+z_2$ in the above computation. Combining both of them, we obtain $\sum\limits_{i \in I}\langle z_1,y_i \rangle \overline{\langle z_2,x_i \rangle}=0, \forall\ z_1 \in \overline{span}\{w_j\}_{j \in I}$ and $z_2 \in (span\{w_j\})^{\perp}$, from which it follows that $\sum\limits_{i \in I}\langle z_1,y_i \rangle x_i=0, \forall\ z_1 \in \overline{span}\{w_j\}_{j \in I}$. Therefore, $w_j=\sum\limits_{i \in I}\langle f_i,u_j\rangle v_i,\ j \in I$.
Further, for $i,j \in I$, consider
\begin{align*}
\sum_{k \in I}\langle v_k,v_i\rangle \langle f_k,u_j\rangle&=\sum_{k \in I}\langle y_k+x_k,y_i+x_i \rangle \langle w_j,y_k\rangle=\sum_{k \in I}\langle y_k,y_i \rangle \langle w_j,y_k\rangle+\sum_{k \in I}\langle x_k,x_i \rangle \langle w_j,y_k\rangle\\
&=\left\langle\sum_{k \in I}\langle w_j,y_k\rangle y_k,y_i \right\rangle+\left\langle\sum_{k \in I}\langle w_j,y_k\rangle x_k,x_i \right\rangle=\langle w_j,y_i\rangle=\langle f_i,u_j\rangle,
\end{align*} which gives $(G(v,v)^t-\mathcal{I})G(f,u)=\bold{0}$, thereby proving that $w$ is a weak R-dual of $f$ with respect to $u$ and $v$.
\qed\end{proof}

In certain circumstances, the Parseval frame $v$ can be written explicitly, as discussed in the corollary below, the proof of which is obvious.
\begin{cor}\label{cor1}
Let $f, u$ and $v$ be as in the hypothesis of the above theorem. Suppose $w$ is a weak R-dual of $f$ with respect to $u$ and $v$. If $w$ is a frame for $H$, then $v_i=\sum\limits_{k \in I}\big\langle u_k,f_i \big\rangle \widetilde{w}_k ,\ i \in I$.
\end{cor}

The following theorem provides necessary and sufficient conditions for the existence of a weak R-dual with respect to a Parseval frame and an orthonormal basis.
\begin{thm}\label{thm2} Let $w=\{w_j\}_{j \in I}$ be a frame sequence in $H$, $f=\{f_i\}_{i \in I}$ be a frame for $H$ and $u=\{u_i\}_{i \in I}$ be a Parseval frame for $H$. Then, there exists an orthonormal basis $v=\{v_i\}_{i \in I}$ such that $w$ is a weak R-dual of $f$ with respect to the Parseval frame $u$ and the orthonormal basis $v$ if and only if the sequence $\{y_i\}_{i \in I}$, given by (\ref{def of ni}), is a Parseval frame for $\overline{span}\{w_j\}_{j \in I}$, $dim(Ker(T_y))=dim((span\{w_j\})^\perp)$ and (\ref{cond1}) holds true.
\end{thm}
\begin{proof}
Suppose there exists an orthonormal basis $v=\{v_i\}_{i \in I}$ such that $w$ is a weak R-dual of $f$ with respect to $u$ and $v$. Then, by Theorem \ref{thm1}, (\ref{cond1}) holds and $v_i=y_i+x_i$, with $y_i$ as defined in (\ref{def of ni}) and $x_i \in (span\{w_j\})^\perp$. For $g \in \overline{span}\{w_j\}$,
\begin{equation*}
\sum_{i \in I}\big|\langle g,y_i \rangle\big|^2=\sum_{i \in I}\big|\langle g,v_i \rangle\big|^2=\big\|g\big\|^2,
\end{equation*}
thereby proving that $\{y_i\}_{i \in I}$ is a Parseval frame for $\overline{span}\{w_j\}$. Further, taking $P$ to be the orthogonal projection of $H$ onto $\overline{span}\{w_j\}$, we get $\sum\limits_{i \in I}c_iy_i=\sum\limits_{i \in I}c_iPv_i=P\left(\sum\limits_{i \in I}c_iv_i\right)$, which implies that $\{c_i\}_{i \in I} \in Ker(T_y)$ if and only if $\sum\limits_{i \in I}c_iv_i \in (span\{w_j\})^\perp$. Hence, $dim(Ker(T_y))=dim((span\{w_j\})^\perp)$.
\par
Conversely, let $\{y_i\}$ be a Parseval frame for $\overline{span}\{w_j\}$, $dim(Ker(T_y))=dim((span\{w_j\})^\perp)$ and (\ref{cond1}) hold true. Then, by Theorem $2$ of \cite{chuang2015equivalent}, there exists an orthonormal basis $v=\{v_i\}_{i \in I}$ for $H$ such that $Pv_i=y_i$, where $P$ is the orthogonal projection of $H$ onto $\overline{span}\{w_j\}$. In other words, we may write $v_i=y_i+x_i$, where $x_i \in (span\{w_j\})^\perp$. Hence, by Theorem \ref{thm1}, $w$ is a weak R-dual of $f$ with respect to the Parseval frame $u$ and the orthonormal basis $v$, thereby proving the theorem.
\qed\end{proof}

Theorem \ref{thm2} shows that one of the necessary conditions for weak R-duality with respect to a Parseval frame and an orthonormal basis is $dim((span\{w_j\})^\perp)=dim(Ker(T_y))$. In the case when the weak R-dual is considered with respect to two Parseval frames say $u$ and $v$, with $v$ not being an orthonormal basis, this condition on dimension need not hold. This can be seen from the theorem below.
\begin{thm}\label{thm4} Let $f=\{f_i\}_{i \in I}$ be a frame, $u=\{u_i\}_{i \in I}$ be a Parseval frame and $v=\{v_i\}_{i \in I}$ be a Parseval frame but not an orthonormal basis for $H$. Suppose $w=\{w_j\}_{j \in I}$ is a weak R-dual of $f$ with respect to $u$ and $v$. Then,
\begin{enumerate}
\item[(i)] $dim((span\{w_j\})^\perp)\leq dim(Ker(T_y))$.
\item[(ii)] If $dim((span\{w_j\})^\perp)<\infty$, then $dim((span\{w_j\})^\perp)<dim(Ker(T_y))$, where $y=\{y_i\}_{i \in I}$ is as defined in (\ref{def of ni}).
\end{enumerate}
\end{thm}
\begin{proof}
The sequence $w$, being a weak R-dual of a frame, is a frame sequence. Suppose the dimension of $(span\{w_j\})^\perp$ is 0. Then, (i) is obviously true. Also, by Corollary \ref{cor1}, we have $v=y$. Therefore, $y$ is a Parseval frame for $H$. If $dim(Ker(T_y))=0$, then $y$ is an orthonormal basis, which contradicts the assumption on $v$. Thus, (ii) is proved when $dim((span\{w_j\})^\perp)=0$. So, we shall now assume that $dim((span\{w_j\})^\perp)>0$.
\par
We denote $Ker^*(T_f)$ to be the set $\{\overline{\bold{c}}: \bold{c} \in Ker(T_f)\}$. By Lemma $6$ in \cite{li2020class} we have, $x \in (span\{w_j\})^\perp$ if and only if $\{\langle x,v_i \rangle \}_{i \in I} \in Ker^*(T_f)$. So, we may define a map $\theta:(span\{w_j\})^\perp\\ \longrightarrow Ker^*(T_f)$ by $\theta(x)=\{\langle x,v_i \rangle \}_{i \in I}$. Clearly, the map is a linear isometry, from which we may infer that $dim((span\{w_j\})^\perp)\leq dim(Ker^*(T_f))$. Now, by using the definition of $y$ we have,
\begin{equation}\label{eq4}
\sum\limits_{i \in I}c_iy_i=\sum\limits_{i \in I}c_i\left(\sum\limits_{k \in I}\big\langle u_k,f_i \big\rangle \widetilde{w}_k\right)=\sum\limits_{k \in I}\left\langle u_k,\sum\limits_{i \in I}\overline{c_i}f_i \right\rangle \widetilde{w}_k.
\end{equation}
The change in the order of summation in the above line and in the subsequent lines of the proof is possible by applying Theorem 1 of \cite{swartz1992iterated}. In fact, taking $\{i_l\}_{l \in I}$ to be an increasing sequence of positive integers, we can obtain the convergence of $\sum\limits_{k \in I}\left\langle u_k,\sum\limits_{l \in I}\overline{c_{i_l}}f_{i_l} \right\rangle\widetilde{w}_k$, by showing that $\left\{\left\langle u_k,\sum\limits_{l \in I}\overline{c_{i_l}}f_{i_l} \right\rangle\right\}_{k \in I}$ is in $\ell^2(I)$, which can be seen as follows. Consider,
\begin{align*}
\left(\sum_{k \in I}\left|\left\langle u_k,\sum_{l \in I}\overline{c_{i_l}}f_{i_l}\right\rangle\right|^2\right)^{\frac{1}{2}}&=\left\|\sum_{l \in I}\overline{c_{i_l}}f_{i_l}\right\|=\sup_{\|x\|=1}\left|\left\langle x,\sum_{l \in I}\overline{c_{i_l}}f_{i_l}\right\rangle \right|
=\sup_{\|x\|=1}\left|\sum_{l \in I}c_{i_l}\langle x,f_{i_l}\rangle \right| \\
&\leq \sup_{\|x\|=1}\left(\sum_{l \in I}|c_{i_l}|^2\right)^\frac{1}{2}\left(\sum_{l \in I}|\langle x,f_{i_l}\rangle|^2\right)^\frac{1}{2}
\leq \sqrt{B_f}\sup_{\|x\|=1}\|\bold{c}\| \|x\|< \infty,
\end{align*}
where $B_f$ is an upper frame bound of $f$. It then follows from (\ref{eq4}) that $Ker^*(T_f)\subset Ker(T_y)$. Also, by Theorem \ref{thm1}, (\ref{cond1}) holds. Subsequently, we have $\big\langle f_i,u_k \big\rangle=\big\langle w_k,y_i \big\rangle$, as was shown in the proof of Theorem \ref{thm1}. Then, $\sum\limits_{i \in I}c_if_i=\sum\limits_{i \in I}c_i\left(\sum\limits_{k \in I}\big\langle f_i,u_k \big\rangle u_k\right)=\sum\limits_{i \in I}c_i\left(\sum\limits_{k \in I}\big\langle w_k,y_i \big\rangle u_k\right)=\sum\limits_{k \in I}\left\langle w_k,\sum\limits_{i \in I}\overline{c_i}y_i \right\rangle u_k$ and so, $Ker(T_y)\subset Ker^*(T_f)$. Thus, $Ker(T_y)=Ker^*(T_f)$, thereby proving (i).
\par
In order to prove (ii), we shall show that the map $\theta$ is not surjective. Suppose, on the contrary, it is surjective. Then, for every $\{c_i\}_{i \in I} \in Ker^*(T_f)$, there exists some $x \in (span\{w_j\})^\perp$ such that $\{\langle x,v_i \rangle \}_{i \in I}=\{c_i\}_{i \in I}$, which implies that $\sum\limits_{j \in I}c_jv_j=x$. Therefore, $\left\{\left\langle \sum\limits_{j \in I}c_jv_j,v_i \right\rangle \right\}_{i \in I}\\ =\{c_i\}_{i \in I}$, which in turn gives $\sum\limits_{j \in I}c_j\left(\langle v_j,v_i \rangle-\delta_{ji}\right)=0,\  \forall\ i \in I$. In other words, for $i \in I$, $\left\langle \{c_j\}_{j \in I},\{\langle v_i,v_j \rangle-\delta_{ji}\}_{j \in I} \right\rangle=0,\  \forall\ \{c_j\}_{j \in I} \in Ker^*(T_f)$. If we show that $\{\langle v_i,v_j \rangle-\delta_{ji}\}_{j \in I} \in Ker^*(T_f), \forall\ i \in I$, then we would get that $v$ is orthonormal, which is a contradiction to our hypothesis. Hence, all that remains to prove is $\{\langle v_i,v_j \rangle-\delta_{ji}\}_{j \in I} \in Ker^*(T_f), \forall\ i \in I$. As $w$ is a weak R-dual of $f$ with respect to $u$ and $v$, we have $(G(v,v)^t-\mathcal{I})G(f,u)=\bold{0}$, from which we get $\left\langle \sum\limits_{j \in I}\langle v_j,v_i \rangle f_j-f_i,u_k \right\rangle=0,\ \forall\ k \in I$. So, $\sum\limits_{j \in I}\langle v_j,v_i \rangle f_j-f_i=0$, thereby proving $\{\langle v_i,v_j \rangle-\delta_{ji}\}_{j \in I} \in Ker^*(T_f), \forall\ i \in I$ and hence the theorem.
\qed\end{proof}

From Theorems \ref{thm2} and \ref{thm4}, we understand that $dim((span\{w_j\})^\perp)\leq dim(Ker(T_y))$ is a necessary condition for $w$ to be a weak R-dual of $f$. The following theorem provides a few sufficient conditions, besides the above condition on the dimension, for $w$ to be a weak R-dual of $f$. However, the results vary depending on the relation between the dimensions, which we elaborate by taking into account all possible cases.
\begin{thm}\label{thm7}
Let $w=\{w_j\}_{j \in I}$ be a frame sequence, $f=\{f_i\}_{i \in I}$ be a frame and $u=\{u_i\}_{i \in I}$ be a Parseval frame for $H$.
\begin{enumerate}
\item[(i)] If $y$, given by (\ref{def of ni}), is a Parseval frame for $\overline{span}\{w_j\}$, (\ref{cond1}) holds and $dim((span\{w_j\})^\perp)$ $<dim(Ker(T_y))$ or $dim((span\{w_j\})^\perp)=dim(Ker(T_y))=\infty$, then there exists a Parseval frame $v=\{v_i\}_{i \in I}$ for $H$, which is not an orthonormal basis, such that $w$ is a weak R-dual of $f$ with respect to $u$ and $v$.
\item[(ii)] If $y$, given by (\ref{def of ni}), is a Parseval frame for $\overline{span}\{w_j\}$, (\ref{cond1}) holds and $dim((span\{w_j\})^\perp)$ $=dim(Ker(T_y))<\infty$, then there exists an orthonormal basis $v=\{v_i\}_{i \in I}$ for $H$ such that $w$ is a weak R-dual of $f$ with respect to $u$ and $v$. Moreover, $w$ cannot be a weak R-dual of $f$ with respect to $u$ and any Parseval frame $v^\prime$, which is not an orthonormal basis.
\item[(iii)] If $dim((span\{w_j\})^\perp)>dim(Ker(T_y))$, then $w$ cannot be a weak R-dual of $f$ with respect to $u$ and any Parseval frame $v$.
\end{enumerate}
\end{thm}
\begin{proof}
Suppose the hypothesis of (i) holds true. Now, $dim(Ker(T_y))=dim((T_y^*(\overline{span}\{w_j\}))^\perp)$. For, $\sum\limits_{i \in I}\overline{c_i}\langle g,y_i\rangle=\left\langle g,\sum\limits_{i \in I}c_iy_i\right\rangle,\ \forall\ g \in \overline{span}\{w_j\}_{j \in I}$ and $\{c_i\}_{i \in I} \in \ell^2(I)$, which implies that $\{c_i\}_{i \in I} \in (T_y^*(\overline{span}\{w_j\}))^\perp$ if and only if $\{c_i\}_{i \in I} \in Ker(T_y)$. Thus, either $dim((span\{w_j\})^\perp)<dim((T_y^*(\overline{span}\{w_j\}))^\perp)$ or $dim((span\{w_j\})^\perp)=dim((T_y^*(\overline{span}\{w_j\}))^\perp)=\infty$.
\par
Let us first consider the case when $m=dim((span\{w_j\})^\perp)<dim((T_y^*(\overline{span}\{w_j\}))^\perp)=d$ (finite or infinite). If $m=0$, then $y$ is a Parseval frame for $H$ and by Theorem \ref{thm1}, $w$ is a weak R-dual of $f$ with respect to $u$ and $y$. However, $y$ is not orthonormal as $dim(Ker(T_y))>0$. Suppose, $m\neq 0$. In order to construct $v$, we define an operator $Q:(span\{w_j\})^\perp\longrightarrow (T_y^*(\overline{span}\{w_j\}))^\perp$ by $Q(\phi_j)=\psi_j,\ 1\leq j\leq m$, where $\{\phi_j\}_{j=1}^m$ and $\{\psi_j\}_{j=1}^d$ are orthonormal bases of $(span\{w_j\})^\perp$ and $(T_y^*(\overline{span}\{w_j\}))^\perp$ respectively. The operator $Q$, extended linearly to $(span\{w_j\})^\perp$, is a non-surjective linear isometry. Using $Q$, we define an operator $U:H\longrightarrow \ell^2(I)$ by $U(z_1+z_2)=T_y^*(z_1)+Q(z_2)$, where $z_1 \in \overline{span}\{w_j\}$ and $z_2 \in (span\{w_j\})^\perp$. One can easily verify that $U$ is also a non-surjective linear isometry. Let $v_i:=U^*(e_i),\ i \in I$, where $\{e_i\}_{i \in I}$ is the standard orthonormal basis of $\ell^2(I)$. It can be easily seen that for $\{c_i\}_{i \in I} \in T_y^*(\overline{span}\{w_j\})$ and $\{d_i\}_{i \in I} \in (T_y^*(\overline{span}\{w_j\}))^\perp, U^*(\{c_i\}+\{d_i\})=T_y(\{c_i\})+Q^*(\{d_i\})$. We may write $v_i=U^*(T_y^*(y_i)+e_i-T_y^*(y_i))$. Clearly, $T_y^*(y_i) \in T_y^*(\overline{span}\{w_j\})$. For $g \in \overline{span}\{w_j\}$,
\begin{equation*}
\langle e_i-T_y^*(y_i),T_y^*(g)\rangle=\langle y_i,g\rangle-\sum_{j \in I}\langle y_i,y_j\rangle \langle y_j,g\rangle=\langle y_i,g\rangle-\left\langle \sum_{j \in I}\langle y_i,y_j\rangle y_j,g\right\rangle=0,
\end{equation*}
as $y$ is a Parseval frame for $\overline{span}\{w_j\}$. So, $e_i-T_y^*(y_i) \in (T_y^*(\overline{span}\{w_j\}))^\perp$. Therefore, $v_i=T_y(T_y^*(y_i))+Q^*(e_i-T_y^*(y_i))=y_i+Q^*(e_i-T_y^*(y_i))$.
Also, $v=\{v_i\}_{i \in I}$ is a Parseval frame for $H$, as $\sum\limits_{i \in I}\left|\langle z,U^*(e_i)\rangle\right|^2=\sum\limits_{i \in I}\left|\langle U(z),e_i\rangle\right|^2=\|U(z)\|^2=\|z\|^2,\ \forall\ z \in H$. However, $v$ is not an orthonormal basis, for otherwise, $U^*$ and hence $U$ will be unitary, which is not true. Applying Theorem \ref{thm1}, we infer that $w$ is a weak R-dual of $f$ with respect to the Parseval frames $u$ and $v$.
\par
In the other case when $dim((span\{w_j\})^\perp)=dim((T_y^*(\overline{span}\{w_j\}))^\perp)=\infty$, we shall define the operator $Q:(span\{w_j\})^\perp\longrightarrow (T_y^*(\overline{span}\{w_j\}))^\perp$ by $Q(\phi_j)=\psi_{j+1}$, where $\{\phi_j\}_{j \in I}$ and $\{\psi_j\}_{j \in I}$ are orthonormal bases of $(span\{w_j\})^\perp$ and $(T_y^*(\overline{span}\{w_j\}))^\perp$ respectively. The operator $Q$, extended linearly to $(span\{w_j\})^\perp$, is again a non-surjective linear isometry. Now, following the lines of proof of the previous case, we can say that there exists a Parseval frame, not an orthonormal basis, $v$ for $H$ such that $w$ is a weak R-dual of $f$ with respect to $u$ and $v$, thereby proving (i).
\par
If the sequence $\{y_i\}_{i \in I}$, given by (\ref{def of ni}), is a Parseval frame for $\overline{span}\{w_j\}_{j \in I}$, (\ref{cond1}) holds and $dim((span\{w_j\})^\perp)=dim(Ker(T_y))<\infty$, then by Theorem \ref{thm2}, there exists an orthonormal basis $v=\{v_i\}_{i \in I}$ for $H$ such that $w$ is a weak R-dual of $f$ with respect to $u$ and $v$. Further, suppose that there exists a Parseval frame $v^\prime$ for $H$ (not an orthonormal basis) such that $w$ is a weak R-dual of $f$ with respect to $u$ and $v^\prime$. Then, by Theorem \ref{thm4}, we get $dim((span\{w_j\})^\perp)<dim(Ker(T_y))$, which contradicts our assumption and hence, such a weak R-dual does not exist.
\par
When $dim((span\{w_j\})^\perp)>dim(Ker(T_y))$, $w$ cannot be a weak R-dual of $f$ with respect to $u$ and any Parseval frame $v$, be it an orthonormal basis or not, by Theorems \ref{thm2} and \ref{thm4}.
\qed\end{proof}

We shall now illustrate the results in the above theorem using the examples below.
First, we shall see an example of a weak R-dual $\{w_j\}_{j \in I}$ of $\{f_i\}_{i \in I}$ when $dim((span\{w_j\})^\perp)<dim(Ker(T_y))$.
\begin{ex}
Let $\{z_i\}_{i \in \mathbb{N}}$ be an orthonormal basis for $H$,
\begin{align*}
\{u_i\}_{i \in \mathbb{N}}=\{f_i\}_{i \in \mathbb{N}}=\left\{\frac{z_1}{\sqrt{2}},\frac{z_1}{\sqrt{2}},\frac{z_2}{\sqrt{2}},\frac{z_2}{\sqrt{2}},...\right\} \hspace{0.7em} and \hspace{0.7em}
\{w_i\}_{i \in \mathbb{N}}=\left\{\frac{z_2}{\sqrt{2}},\frac{z_2}{\sqrt{2}},\frac{z_3}{\sqrt{2}},\frac{z_3}{\sqrt{2}},...\right\}.
\end{align*}
Clearly, $\{u_i\}_{i \in \mathbb{N}}$ and $\{f_i\}_{i \in \mathbb{N}}$ are Parseval frames for $H$, and $\{w_i\}_{i \in \mathbb{N}}$ is a frame sequence in $H$ with its canonical dual being itself. We shall compute $y_i$'s as follows. For $i \in \mathbb{N}$,
\begin{align*}
y_{2i}&=\sum_{k \in \mathbb{N}}\langle u_{2k-1},f_{2i} \rangle w_{2k-1}+\sum_{k \in \mathbb{N}}\big\langle u_{2k},f_{2i} \big\rangle w_{2k}\\
&=\sum_{k \in \mathbb{N}}\left\langle \frac{z_k}{\sqrt{2}},\frac{z_i}{\sqrt{2}} \right\rangle \frac{z_{k+1}}{\sqrt{2}}+\sum_{k \in \mathbb{N}}\left\langle \frac{z_k}{\sqrt{2}},\frac{z_i}{\sqrt{2}} \right\rangle \frac{z_{k+1}}{\sqrt{2}}\\
&=\frac{z_{i+1}}{\sqrt{2}}.
\end{align*}
Similarly, $y_{2i-1}=\frac{z_{i+1}}{\sqrt{2}},\ i \in \mathbb{N}$. So, $\{y_i\}_{i \in \mathbb{N}}=\left\{\frac{z_2}{\sqrt{2}},\frac{z_2}{\sqrt{2}},\frac{z_3}{\sqrt{2}},\frac{z_3}{\sqrt{2}},...\right\}$, which is a Parseval frame for $\overline{span}\{w_j\}_{j \in \mathbb{N}}$. In fact, in this case $dim((span\{w_j\})^\perp)<dim(Ker(T_y))=\infty$. Further, for all odd $i,j$ we have,
\begin{align*}
\sum\limits_{k \in \mathbb{N}}\big\langle w_k,w_i \big\rangle \big\langle u_k,f_j \big\rangle&=\sum\limits_{k \in \mathbb{N}}\big\langle w_{2k-1},w_i \big\rangle \big\langle u_{2k-1},f_j \big\rangle+\sum\limits_{k \in \mathbb{N}}\big\langle w_{2k},w_i \big\rangle \big\langle u_{2k},f_j \big\rangle\\
&=\sum\limits_{k \in \mathbb{N}}\left\langle \frac{z_{k+1}}{\sqrt{2}},\frac{z_{(i+3)/2}}{\sqrt{2}} \right\rangle \left\langle \frac{z_k}{\sqrt{2}},\frac{z_{(j+1)/2}}{\sqrt{2}} \right\rangle\\ &\qquad \quad+\sum\limits_{k \in \mathbb{N}}\left\langle \frac{z_{k+1}}{\sqrt{2}},\frac{z_{(i+3)/2}}{\sqrt{2}} \right\rangle \left\langle \frac{z_k}{\sqrt{2}},\frac{z_{(j+1)/2}}{\sqrt{2}} \right\rangle\\
&=\left\langle \frac{z_{(i+1)/2}}{\sqrt{2}},\frac{z_{(j+1)/2}}{\sqrt{2}} \right\rangle\\
&=\big\langle u_i,f_j \big\rangle.
\end{align*}
Similarly, we can prove that $\sum\limits_{k \in \mathbb{N}}\big\langle w_k,w_i \big\rangle \big\langle u_k,f_j \big\rangle=\big\langle u_i,f_j \big\rangle$ for all the other choices of $i,j$, which gives $(G(\widetilde{w},w)^t-\mathcal{I})G(u,f)=\bold{0}$. Now, let $\{x_i\}_{i \in \mathbb{N}}=\left\{\frac{z_1}{\sqrt{2}},\frac{-z_1}{\sqrt{2}},0,0,...\right\}$, which is a sequence in $(span\{w_j\})^\perp$. Then, taking $\{v_i\}_{i \in \mathbb{N}}:=\{y_i+x_i\}_{i \in \mathbb{N}}=\left\{\frac{z_2+z_1}{\sqrt{2}},\frac{z_2-z_1}{\sqrt{2}},\frac{z_3}{\sqrt{2}},\frac{z_3}{\sqrt{2}},...\right\}$, we have $\{v_i\}_{i \in \mathbb{N}}$ as a Parseval frame for $H$ but not an orthonormal basis. We shall show that $\{w_i\}_{i \in \mathbb{N}}$ is a weak R-dual of $\{f_i\}_{i \in \mathbb{N}}$ with respect to the Parseval frames $\{u_i\}_{i \in \mathbb{N}}$ and $\{v_i\}_{i \in \mathbb{N}}$. Consider,
\begin{align*}
\sum_{i \in \mathbb{N}}\langle f_i,u_j\rangle v_i&=\sum_{i=1}^2\langle f_i,u_j\rangle v_i+\sum_{i=2}^\infty \langle f_{2i-1},u_j\rangle v_{2i-1}+\sum_{i=2}^\infty \langle f_{2i},u_j\rangle v_{2i}\\
&=\begin{cases}
\left\langle \frac{z_1}{\sqrt{2}},\frac{z_1}{\sqrt{2}}\right\rangle \left(\frac{z_2+z_1}{\sqrt{2}}\right)+\left\langle \frac{z_1}{\sqrt{2}},\frac{z_1}{\sqrt{2}}\right\rangle \left(\frac{z_2-z_1}{\sqrt{2}}\right), & j=1,2\\
2\sum\limits_{i=2}^\infty\left\langle \frac{z_i}{\sqrt{2}},\frac{z_{(j+1)/2}}{\sqrt{2}}\right\rangle \frac{z_{i+1}}{\sqrt{2}}, & j\ is\ odd,\ j\geq 3\\
2\sum\limits_{i=2}^\infty\left\langle \frac{z_i}{\sqrt{2}},\frac{z_{j/2}}{\sqrt{2}}\right\rangle \frac{z_{i+1}}{\sqrt{2}}, & j\ is\ even,\ j\geq 3
\end{cases}\\
&=\begin{cases}
\frac{z_2}{\sqrt{2}}, & j=1,2\\
\frac{z_{(j+3)/2}}{\sqrt{2}}, & j\ is\ odd,\ j\geq 3\\
\frac{z_{(j+2)/2}}{\sqrt{2}}, & j\ is\ even,\ j\geq 3\\
\end{cases}\\
&=w_j
\end{align*} and
\begin{align*}
\sum_{k \in \mathbb{N}}\langle v_k,v_i\rangle \langle f_k,u_j\rangle&=\sum_{k=1}^2 \langle v_k,v_i\rangle \langle f_k,u_j\rangle+\sum_{k=2}^\infty \langle v_{2k-1},v_i\rangle \langle f_{2k-1},u_j\rangle+\sum_{k=2}^\infty\langle v_{2k},v_i\rangle \langle f_{2k},u_j\rangle\\
&=\begin{cases}
\left\langle \frac{z_2+z_1}{\sqrt{2}},v_i\right\rangle \langle \frac{z_1}{\sqrt{2}},u_j\rangle+\left\langle \frac{z_2-z_1}{\sqrt{2}},v_i\right\rangle \langle \frac{z_1}{\sqrt{2}},u_j\rangle, & i=1,2\\
2\sum\limits_{k=2}^\infty\left\langle \frac{z_{k+1}}{\sqrt{2}},\frac{z_{(i+3)/2}}{\sqrt{2}}\right\rangle \langle \frac{z_k}{\sqrt{2}},u_j\rangle, & i\ is\ odd,\ i\geq 3\\
2\sum\limits_{k=2}^\infty\left\langle \frac{z_{k+1}}{\sqrt{2}},\frac{z_{(i+2)/2}}{\sqrt{2}}\right\rangle \langle \frac{z_k}{\sqrt{2}},u_j\rangle, & i\ is\ even,\ i\geq 3
\end{cases}\\
&=\begin{cases}
\left\langle \frac{z_1}{\sqrt{2}},u_j\right\rangle, & i=1,2\\
\left\langle \frac{z_{(i+1)/2}}{\sqrt{2}},u_j\right\rangle, & i\ is\ odd,\ i\geq 3\\
\left\langle \frac{z_{i/2}}{\sqrt{2}},u_j\right\rangle, & i\ is\ even,\ i\geq 3
\end{cases}\\
&=\langle f_i,u_j\rangle.
\end{align*}
\end{ex}

\vskip 1em
An example of a weak R-dual with respect to two Parseval frames, which are not orthonormal bases, when $dim((span\{w_j\})^\perp)=dim(Ker(T_y))=\infty$ is as follows.
\begin{ex}
Consider the sequences
\begin{align*}
\{f_i\}_{i \in \mathbb{N}}=\{u_i\}_{i \in \mathbb{N}}&=\left\{\frac{z_1}{2},\frac{z_1}{2},\frac{z_1}{2},\frac{z_1}{2},\frac{z_2}{2},\frac{z_2}{2},\frac{z_2}{2},\frac{z_2}{2},...\right\} \\
and \hspace{1em} \{w_i\}_{i \in \mathbb{N}}&=\left\{\frac{z_1}{2},\frac{z_1}{2},\frac{z_1}{2},\frac{z_1}{2},\frac{z_3}{2},\frac{z_3}{2},\frac{z_3}{2},\frac{z_3}{2},...\right\},
\end{align*} where $\{z_i\}_{i \in \mathbb{N}}$ is an orthonormal basis for $H$. It can be easily verified that $\{u_i\}_{i \in \mathbb{N}}$ is a Parseval frame for $H, \{f_i\}_{i \in \mathbb{N}}$ is a frame for $H, \{w_i\}_{i \in \mathbb{N}}$ is a frame sequence in $H$ and the canonical dual $\{\widetilde{w}_i\}_{i \in \mathbb{N}}$ of $\{w_i\}_{i \in \mathbb{N}}$ is itself. The terms of the sequence $\{y_i\}_{i \in \mathbb{N}}$ can be explicitly found as follows. For $i \in \mathbb{N}$,
\begin{align*}
y_{4i-3}&=\sum_{k \in \mathbb{N}}\big\langle u_{4k-3},f_{4i-3} \big\rangle w_{4k-3}+\sum_{k \in \mathbb{N}}\big\langle u_{4k-2},f_{4i-3} \big\rangle w_{4k-2}+\sum_{k \in \mathbb{N}}\big\langle u_{4k-1},f_{4i-3} \big\rangle w_{4k-1}\\ & \qquad+\sum_{k \in \mathbb{N}}\big\langle u_{4k},f_{4i-3} \big\rangle w_{4k}\\
&=\sum_{k \in \mathbb{N}}\left\langle \frac{z_k}{2},\frac{z_i}{2} \right\rangle \frac{z_{2k-1}}{2}+\sum_{k \in \mathbb{N}}\left\langle \frac{z_k}{2},\frac{z_i}{2} \right\rangle \frac{z_{2k-1}}{2}+\sum_{k \in \mathbb{N}}\left\langle \frac{z_k}{2},\frac{z_i}{2} \right\rangle \frac{z_{2k-1}}{2}+\sum_{k \in \mathbb{N}}\left\langle \frac{z_k}{2},\frac{z_i}{2} \right\rangle \frac{z_{2k-1}}{2}\\
&=4 \sum_{k \in \mathbb{N}}\left\langle \frac{z_k}{2},\frac{z_i}{2} \right\rangle \frac{z_{2k-1}}{2}\\
&=\frac{z_{2i-1}}{2}.
\end{align*}
Likewise, $y_{4i-2}=y_{4i-1}=y_{4i}=\frac{z_{2i-1}}{2},\ i \in \mathbb{N}$. Therefore, $\{y_i\}_{i \in \mathbb{N}}=\left\{\frac{z_1}{2},\frac{z_1}{2},\frac{z_1}{2},\frac{z_1}{2},\frac{z_3}{2},\frac{z_3}{2},\frac{z_3}{2},\frac{z_3}{2},...\right\}$, which is a Parseval frame for $\overline{span}\{w_j\}_{j \in \mathbb{N}}$. Furthermore, for all $j \in \mathbb{N}$ and $i=4m-3,\ m \in \mathbb{N}$,
\begin{align*}
\sum\limits_{k \in \mathbb{N}}\big\langle w_k,w_i \big\rangle \big\langle u_k,f_j \big\rangle&=\sum\limits_{k \in \mathbb{N}}\big\langle w_{4k-3},w_i \big\rangle \big\langle u_{4k-3},f_j \big\rangle+\sum\limits_{k \in \mathbb{N}}\big\langle w_{4k-2},w_i \big\rangle \big\langle u_{4k-2},f_j \big\rangle\\ & \qquad+\sum\limits_{k \in \mathbb{N}}\big\langle w_{4k-1},w_i \big\rangle \big\langle u_{4k-1},f_j \big\rangle+\sum\limits_{k \in \mathbb{N}}\big\langle w_{4k},w_i \big\rangle \big\langle u_{4k},f_j \big\rangle\\
&=\sum\limits_{k \in \mathbb{N}}\left\langle \frac{z_{2k-1}}{2},\frac{z_{(i+1)/2}}{2} \right\rangle \left\langle \frac{z_k}{2},f_j \right\rangle+\sum\limits_{k \in \mathbb{N}}\left\langle \frac{z_{2k-1}}{2},\frac{z_{(i+1)/2}}{2} \right\rangle \left\langle \frac{z_k}{2},f_j \right\rangle\\ &\qquad+\sum\limits_{k \in \mathbb{N}}\left\langle \frac{z_{2k-1}}{2},\frac{z_{(i+1)/2}}{2} \right\rangle \left\langle \frac{z_k}{2},f_j \right\rangle+\sum\limits_{k \in \mathbb{N}}\left\langle \frac{z_{2k-1}}{2},\frac{z_{(i+1)/2}}{2} \right\rangle \left\langle \frac{z_k}{2},f_j \right\rangle\\
&=4\sum\limits_{k \in \mathbb{N}}\left\langle \frac{z_{2k-1}}{2},\frac{z_{(i+1)/2}}{2} \right\rangle \left\langle \frac{z_k}{2},f_j \right\rangle\\
&=\left\langle \frac{z_{(i+3)/4}}{2},f_j \right\rangle\\
&=\big\langle u_i,f_j \big\rangle.
\end{align*}
The above relation holds for the other choices of $i,j$ as well, which implies that $(G(\widetilde{w},w)^t-\mathcal{I})G(u,f)=\bold{0}$. We note that in this example, $dim((span\{w_j\})^\perp)=dim(Ker(T_y))=\infty$. Now, taking $\{x_i\}_{i \in \mathbb{N}}=\left\{\frac{z_2}{2},\frac{z_2}{2},\frac{-z_2}{2},\frac{-z_2}{2},\frac{z_4}{2},\frac{z_4}{2},\frac{-z_4}{2},\frac{-z_4}{2},...\right\}$ in $(span\{w_j\})^\perp$, we get $\{v_i\}_{i \in \mathbb{N}}=\{y_i+x_i\}_{i \in \mathbb{N}}=\left\{\frac{z_1+z_2}{2},\frac{z_1+z_2}{2},\frac{z_1-z_2}{2},\frac{z_1-z_2}{2},\frac{z_3+z_4}{2},\frac{z_3+z_4}{2},\frac{z_3-z_4}{2},\frac{z_3-z_4}{2},...\right\}$, which is a Parseval frame for $H$ but not an orthonormal basis. Consider,
\begin{align*}
\sum_{i \in \mathbb{N}}\langle f_i,u_j\rangle v_i&=\sum_{i \in \mathbb{N}}\big\langle f_{4i-3},u_j \big\rangle v_{4i-3}+\sum_{i \in \mathbb{N}}\big\langle f_{4i-2},u_j \big\rangle v_{4i-2}+\sum_{i \in \mathbb{N}}\big\langle f_{4i-1},u_j \big\rangle v_{4i-1}+\sum_{i \in \mathbb{N}}\big\langle f_{4i},u_j \big\rangle v_{4i}\\
&=2\left(\sum_{i \in \mathbb{N}}\left\langle \frac{z_i}{2},u_j\right\rangle \left(\frac{z_{2i-1}+z_{2i}}{2}\right)+\sum_{i \in \mathbb{N}}\left\langle \frac{z_i}{2},u_j\right\rangle \left(\frac{z_{2i-1}-z_{2i}}{2}\right)\right)\\
&=2\sum_{i \in \mathbb{N}}\left\langle \frac{z_i}{2},u_j\right\rangle z_{2i-1}\\
&=\begin{cases}
2\sum\limits_{i \in \mathbb{N}}\left\langle \frac{z_i}{2},\frac{z_{(j+3)/4}}{2}\right\rangle z_{2i-1}, & j=4k-3,\ k \in \mathbb{N}\\
2\sum\limits_{i \in \mathbb{N}}\left\langle \frac{z_i}{2},\frac{z_{(j+2)/4}}{2}\right\rangle z_{2i-1}, & j=4k-2,\ k \in \mathbb{N}\\
2\sum\limits_{i \in \mathbb{N}}\left\langle \frac{z_i}{2},\frac{z_{(j+1)/4}}{2}\right\rangle z_{2i-1}, & j=4k-1,\ k \in \mathbb{N}\\
2\sum\limits_{i \in \mathbb{N}}\left\langle \frac{z_i}{2},\frac{z_{j/4}}{2}\right\rangle z_{2i-1}, & j=4k,\ k \in \mathbb{N}
\end{cases}\\
&=\begin{cases}
\frac{z_{(j+1)/2}}{2}, & j=4k-3,\ k \in \mathbb{N}\\
\frac{z_{j/2}}{2}, & j=4k-2,\ k \in \mathbb{N}\\
\frac{z_{(j-1)/2}}{2}, & j=4k-1,\ k \in \mathbb{N}\\
\frac{z_{(j-2)/2}}{2}, & j=4k,\ k \in \mathbb{N}
\end{cases}\\
&=w_j
\end{align*} and
\begin{align*}
\sum_{k \in \mathbb{N}}\langle v_k,v_i\rangle \langle f_k,u_j\rangle&=\sum_{k \in \mathbb{N}}\langle v_{4k-3},v_i\rangle \langle f_{4k-3},u_j\rangle+\sum_{k \in \mathbb{N}}\langle v_{4k-2},v_i\rangle \langle f_{4k-2},u_j\rangle\\ & \qquad+\sum_{k \in \mathbb{N}}\langle v_{4k-1},v_i\rangle \langle f_{4k-1},u_j\rangle+\sum_{k \in \mathbb{N}}\langle v_{4k},v_i\rangle \langle f_{4k},u_j\rangle\\
&=2\left(\sum_{k \in \mathbb{N}}\left\langle \frac{z_{2k-1}+z_{2k}}{2},v_i\right\rangle \left\langle \frac{z_k}{2},u_j\right\rangle+\sum_{k \in \mathbb{N}}\left\langle \frac{z_{2k-1}-z_{2k}}{2},v_i\right\rangle \left\langle \frac{z_k}{2},u_j\right\rangle\right)\\
&=2\sum_{k \in \mathbb{N}}\left\langle z_{2k-1},v_i\right\rangle \left\langle \frac{z_k}{2},u_j\right\rangle\\
&=\begin{cases}
2\sum\limits_{k \in \mathbb{N}}\left\langle z_{2k-1},\frac{z_{2m-1}+z_{2m}}{2}\right\rangle \left\langle \frac{z_k}{2},u_j\right\rangle, & i=4m-3,\ m \in \mathbb{N}\\
2\sum\limits_{k \in \mathbb{N}}\left\langle z_{2k-1},\frac{z_{2m-1}+z_{2m}}{2}\right\rangle \left\langle \frac{z_k}{2},u_j\right\rangle, & i=4m-2,\ m \in \mathbb{N}\\
2\sum\limits_{k \in \mathbb{N}}\left\langle z_{2k-1},\frac{z_{2m-1}-z_{2m}}{2}\right\rangle \left\langle \frac{z_k}{2},u_j\right\rangle, & i=4m-1,\ m \in \mathbb{N}\\
2\sum\limits_{k \in \mathbb{N}}\left\langle z_{2k-1},\frac{z_{2m-1}-z_{2m}}{2}\right\rangle \left\langle \frac{z_k}{2},u_j\right\rangle, & i=4m,\ m \in \mathbb{N}
\end{cases}\\
&=\begin{cases}
\left\langle \frac{z_{(i+3)/4}}{2},u_j\right\rangle, & i=4m-3,\ m \in \mathbb{N}\\
\left\langle \frac{z_{(i+2)/4}}{2},u_j\right\rangle, & i=4m-2,\ m \in \mathbb{N}\\
\left\langle \frac{z_{(i+1)/4}}{2},u_j\right\rangle, & i=4m-1,\ m \in \mathbb{N}\\
\left\langle \frac{z_{i/4}}{2},u_j\right\rangle, & i=4m,\ m \in \mathbb{N}
\end{cases}\\
&=\langle f_i,u_j\rangle.
\end{align*}
Hence, $\{w_i\}_{i \in \mathbb{N}}$ is a weak R-dual of $\{f_i\}_{i \in \mathbb{N}}$ with respect to the Parseval frames $\{u_i\}_{i \in \mathbb{N}}$ and $\{v_i\}_{i \in \mathbb{N}}$.
\end{ex}

The following is an example of a weak R-dual with respect to a Parseval frame and an orthonormal basis, when $dim((span\{w_j\})^\perp)=dim(Ker(T_y))<\infty$.
\begin{ex}
Consider the sequences
\begin{align*}
\{f_i\}_{i \in \mathbb{N}}=\{u_i\}_{i \in \mathbb{N}}&=\left\{\frac{z_1}{\sqrt{2}},\frac{z_1}{\sqrt{2}},\frac{z_2}{\sqrt{2}},\frac{z_2}{\sqrt{2}},z_3,z_4,z_5,...\right\} \\
and \hspace{1em} \{w_i\}_{i \in \mathbb{N}}&=\left\{\frac{z_3}{\sqrt2},\frac{z_3}{\sqrt2},\frac{z_4}{\sqrt{2}},\frac{z_4}{\sqrt{2}},z_5,z_6,z_7,...\right\},
\end{align*}
where $\{z_i\}_{i \in \mathbb{N}}$ is an orthonormal basis for $H$. It is clear that $\{u_i\}_{i \in \mathbb{N}}$ and $\{f_i\}_{i \in \mathbb{N}}$ are Parseval frames for $H$. Also, $\{w_i\}_{i \in \mathbb{N}}$ is a frame sequence in $H$ and the canonical dual $\{\widetilde{w}_i\}_{i \in \mathbb{N}}$ of $\{w_i\}_{i \in \mathbb{N}}$ is itself. For $i \in \mathbb{N}$,
\begin{align*}
y_{2i}&=\sum_{k=1}^2\big\langle u_{k},f_{2i} \big\rangle w_{k}+\sum_{k=3}^4\big\langle u_{k},f_{2i} \big\rangle w_{k}+\sum_{k=3}^\infty\big\langle u_{2k-1},f_{2i} \big\rangle w_{2k-1}+\sum_{k=3}^\infty\big\langle u_{2k},f_{2i} \big\rangle w_{2k}\\
&=\begin{cases}
2\left\langle \frac{z_1}{\sqrt{2}},\frac{z_1}{\sqrt{2}}\right\rangle \frac{z_3}{\sqrt{2}}, & i=1\\
2\left\langle \frac{z_2}{\sqrt{2}},\frac{z_2}{\sqrt{2}}\right\rangle \frac{z_4}{\sqrt{2}}, & i=2\\
\sum\limits_{k=3}^\infty\big\langle z_{2k-3},z_{2i-2} \big\rangle z_{2k-1}+\sum\limits_{k=3}^\infty\big\langle z_{2k-2},z_{2i-2} \big\rangle z_{2k}, & i\geq 3
\end{cases}\\
&=\begin{cases}
\frac{z_3}{\sqrt{2}}, & i=1\\
\frac{z_4}{\sqrt{2}}, & i=2\\
z_{2i}, & i\geq 3
\end{cases}.
\end{align*}
By similar computations, $y_{2i-1}=\begin{cases}
\frac{z_3}{\sqrt{2}}, & i=1\\
\frac{z_4}{\sqrt{2}}, & i=2\\
z_{2i-1}, & i\geq 3
\end{cases}$. Thus, $\{y_i\}=\left\{\frac{z_3}{\sqrt2},\frac{z_3}{\sqrt2},\frac{z_4}{\sqrt{2}},\frac{z_4}{\sqrt{2}},z_5,z_6,z_7,...\right\},$ which is a Parseval frame for $\overline{span}\{w_j\}_{j \in \mathbb{N}}$. Moreover, for $j \in \mathbb{N}$ and $i=2m,\ m \in \mathbb{N}$,
\begin{align*}
\sum\limits_{k \in \mathbb{N}}\big\langle w_k,w_i \big\rangle \big\langle u_k,f_j \big\rangle&=\sum\limits_{k=1}^2\big\langle w_{k},w_{2m} \big\rangle \big\langle u_{k},f_j \big\rangle+\sum\limits_{k=3}^4\big\langle w_{k},w_{2m} \big\rangle \big\langle u_{k},f_j \big\rangle\\ &\qquad+\sum\limits_{k=3}^\infty\big\langle w_{2k-1},w_{2m} \big\rangle \big\langle u_{2k-1},f_j \big\rangle+\sum\limits_{k=3}^\infty\big\langle w_{2k},w_{2m} \big\rangle \big\langle u_{2k},f_j \big\rangle\\
&=\begin{cases}
2\left\langle \frac{z_3}{\sqrt{2}},\frac{z_3}{\sqrt{2}}\right\rangle \left\langle \frac{z_1}{\sqrt{2}},f_j\right\rangle, & m=1\\
2\left\langle \frac{z_4}{\sqrt{2}},\frac{z_4}{\sqrt{2}}\right\rangle \left\langle\frac{z_2}{\sqrt{2}},f_j\right\rangle, & m=2\\
\sum\limits_{k=3}^\infty\big\langle z_{2k-1},z_{2m} \big\rangle \langle z_{2k-3},f_j\rangle+\sum\limits_{k=3}^\infty\big\langle z_{2k},z_{2m} \big\rangle \langle z_{2k-2},f_j\rangle, & m\geq 3
\end{cases}\\
&=\begin{cases}
\left\langle \frac{z_1}{\sqrt{2}},f_j\right\rangle, & m=1\\
\left\langle\frac{z_2}{\sqrt{2}},f_j\right\rangle, & m=2\\
\langle z_{2m-2},f_j\rangle, & m\geq 3
\end{cases}\\
&=\langle u_{2m},f_j \rangle\\
&=\big\langle u_i,f_j \big\rangle.
\end{align*}
Similarly, the above relation holds for all $j \in \mathbb{N}$ and $i=2m-1,\ m \in \mathbb{N}$ as well, which implies that (\ref{cond1}) is satisfied. Also, $dim((span\{w_j\})^\perp)=dim(Ker(T_y))=2$. Thus, the hypotheses in (ii) of Theorem \ref{thm7} are satisfied. Now, we shall take $\{x_i\}_{i \in \mathbb{N}}=\left\{\frac{z_1}{\sqrt{2}},\frac{-z_1}{\sqrt{2}},\frac{z_2}{\sqrt{2}},\frac{-z_2}{\sqrt{2}},0,0,0,...\right\}$ in $(span\{w_j\})^\perp$. Then, $\{v_i\}_{i \in \mathbb{N}}:=\{y_i+x_i\}_{i \in \mathbb{N}}=\left\{\frac{z_3+z_1}{\sqrt{2}},\frac{z_3-z_1}{\sqrt{2}},\frac{z_4+z_2}{\sqrt{2}},\frac{z_4-z_2}{\sqrt{2}},z_5,z_6,z_7,...\right\}$ is an orthonormal basis of $H$. Consider,
\begin{align*}
\sum_{i \in \mathbb{N}}\langle f_i,u_j\rangle v_i&=\sum_{i=1}^2\big\langle f_{i},u_j \big\rangle v_{i}+\sum_{i=3}^4\big\langle f_{i},u_j \big\rangle v_{i}+\sum_{i=3}^\infty\big\langle f_{2i-1},u_j \big\rangle v_{2i-1}+\sum_{i=3}^\infty\big\langle f_{2i},u_j \big\rangle v_{2i}\\
&=\begin{cases}
\left\langle \frac{z_1}{\sqrt{2}},\frac{z_1}{\sqrt{2}}\right\rangle \left(\frac{z_3+z_1}{\sqrt{2}}\right)+\left\langle \frac{z_1}{\sqrt{2}},\frac{z_1}{\sqrt{2}}\right\rangle \left(\frac{z_3-z_1}{\sqrt{2}}\right), & j=1,2\\
\left\langle \frac{z_2}{\sqrt{2}},\frac{z_2}{\sqrt{2}}\right\rangle \left(\frac{z_4+z_2}{\sqrt{2}}\right)+\left\langle \frac{z_2}{\sqrt{2}},\frac{z_2}{\sqrt{2}}\right\rangle \left(\frac{z_4-z_2}{\sqrt{2}}\right), & j=3,4\\
\sum\limits_{i=3}^\infty\big\langle z_{2i-3},z_{j-2} \big\rangle z_{2i-1}+\sum\limits_{i=3}^\infty\big\langle z_{2i-2},z_{j-2} \big\rangle z_{2i}, & j\geq 5
\end{cases}\\
&=\begin{cases}
\frac{z_3}{\sqrt{2}}, & j=1,2\\
\frac{z_4}{\sqrt{2}}, & j=3,4\\
z_j, & j\geq 5
\end{cases}\\
&=w_j
\end{align*} and
\begin{align*}
\sum_{k \in \mathbb{N}}\langle v_k,v_i\rangle \langle f_k,u_j\rangle&=\sum_{k=1}^2\langle v_k,v_i\rangle \langle f_k,u_j\rangle+\sum_{k=3}^4\langle v_k,v_i\rangle \langle f_k,u_j\rangle+\sum_{k=3}^\infty\langle v_{2k-1},v_i\rangle \langle f_{2k-1},u_j\rangle\\
&\quad\quad +\sum_{k=3}^\infty\langle v_{2k},v_i\rangle \langle f_{2k},u_j\rangle\\
&=\begin{cases}
\left\langle \frac{z_3+z_1}{\sqrt{2}},v_i\right\rangle \left\langle \frac{z_1}{\sqrt{2}},u_j\right\rangle+\left\langle \frac{z_3-z_1}{\sqrt{2}},v_i\right\rangle \left\langle \frac{z_1}{\sqrt{2}},u_j\right\rangle, & i=1,2\\
\left\langle \frac{z_4+z_2}{\sqrt{2}},v_i\right\rangle \left\langle \frac{z_2}{\sqrt{2}},u_j\right\rangle+\left\langle \frac{z_4-z_2}{\sqrt{2}},v_i\right\rangle \left\langle \frac{z_2}{\sqrt{2}},u_j\right\rangle, & i=3,4\\
\sum\limits_{k=3}^\infty\langle z_{2k-1},z_i\rangle \langle z_{2k-3},u_j\rangle+\sum\limits_{k=3}^\infty\langle z_{2k},z_i\rangle \langle z_{2k-2},u_j\rangle, & i\geq 5
\end{cases}\\
&=\begin{cases}
\left\langle \frac{z_1}{\sqrt{2}},u_j\right\rangle, & i=1,2\\
\left\langle \frac{z_2}{\sqrt{2}},u_j\right\rangle, & i=3,4\\
\langle z_{i-2},u_j\rangle, & i\geq 5
\end{cases}\\
&=\langle f_i,u_j\rangle,
\end{align*}
which shows that $\{w_i\}_{i \in \mathbb{N}}$ is a weak R-dual of $\{f_i\}_{i \in \mathbb{N}}$ with respect to $\{u_i\}_{i \in \mathbb{N}}$ and $\{v_i\}_{i \in \mathbb{N}}$.
\end{ex}

We shall now see another interesting characterization of weak R-duality without involving any condition on the dimensions of $(span\{w_j\})^\perp$ and $Ker(T_y)$. In Theorem \ref{thm7}, we have observed that there are scenarios when a given frame sequence $w$ cannot be a weak R-dual of a given frame $f$. However, the following theorem provides conditions which renders $w$ as a weak R-dual of a frame $f^\prime$, which is obtained by slightly modifying $f$.
For a sequence $h=\{h_i\}_{i \in I}$, we define two sequences, $h^\prime=\{h_i^\prime\}_{i \in I}$ and $h^{\prime\prime}=\{h_i^{\prime\prime}\}_{i \in I}$, associated with $h$, which are given by
\begin{equation*}
h_i^\prime
:= \begin{cases}
h_{\frac{i+1}{2}}, & \text{$i\ is\ odd$}\\
0, & \text{$i\ is\ even$}
\end{cases} \hspace{1em} and \hspace{1em} h_i^{\prime\prime}
:=
\begin{cases}
0, & \text{$i\ is\ odd$}\\
h_{\frac{i}{2}}, & \text{$i\ is\ even$}
\end{cases}.
\end{equation*}

\begin{thm}\label{thm3} Let $w$, $f$ and $u$ be as in the hypothesis of Theorem \ref{thm7}. Then, the sequence $\{y_i\}_{i \in I}$, given by (\ref{def of ni}), is a Parseval frame for $\overline{span}\{w_j\}_{j \in I}$ and (\ref{cond1}) holds if and only if there exists a Parseval frame $v=\{v_i\}_{i \in I}$ for $H$, not necessarily an orthonormal basis, such that $w$ is a weak R-dual of $f^\prime$ with respect to $u$ and $v$.
\end{thm}
\begin{proof}
Suppose $\{y_i\}_{i \in I}$ is a Parseval frame for $\overline{span}\{w_j\}_{j \in I}$ and (\ref{cond1}) holds. Let $\{q_i\}_{i \in I}$ be any Parseval frame for $(span\{w_j\})^\perp$. Define $v=\{v_i\}_{i \in I}=\{y_i^\prime +q_i^{\prime\prime} \}_{i \in I}=
\begin{cases}
y_{\frac{i+1}{2}} , & i\ is\ odd\\
q_{\frac{i}{2}} ,& i\ is\ even
\end{cases}
$. Then, $v$ is a Parseval frame for $H$. In fact, for $z \in H$, we may write $z=z_1+z_2,$ where $z_1 \in \overline{span}\{w_j\}_{j \in I}$ and $z_2 \in (span\{w_j\})^\perp$. So,
\begin{align*}
\sum_{i \in I}|\big \langle z,v_i \big \rangle|^2&=\sum_{i \in I}|\big \langle z_1+z_2,v_{2i-1} \big \rangle|^2+\sum_{i \in I}|\big \langle z_1+z_2,v_{2i} \big \rangle|^2\\
&=\sum_{i \in I}|\big \langle z_1+z_2,y_i \big \rangle|^2+\sum_{i \in I}|\big \langle z_1+z_2,q_i \big \rangle|^2\\
&=\sum_{i \in I}|\big \langle z_1,y_i \big \rangle|^2+\sum_{i \in I}|\big \langle z_2,q_i \big \rangle|^2\\
&=\|z_1\|^2+\|z_2\|^2\\
&=\|z_1+z_2\|^2\\
&=\|z\|^2.
\end{align*} We note that if $\{q_i\}_{i \in I}$ is chosen to be a Parseval frame, which is not an orthonormal basis, then the sequence $v$, so defined, is also not an orthonormal basis. Hence, by Theorem \ref{thm1}, $w$ is a weak R-dual of $f^\prime$ with respect to the Parseval frames $u$ and $v$.
\par
Conversely, let us suppose that there exists a Parseval frame $v=\{v_i\}_{i \in I}$ for $H$ such that $w$ is a weak R-dual of $f^\prime$ with respect to the Parseval frames $u$ and $v$. Then, by Theorem \ref{thm1}, $(G(\widetilde{w},w)^t-\mathcal{I})G(u,f^\prime)=\bold{0}$ and obviously (\ref{cond1}) is also true. Moreover, $v_i=y_i^\prime+x_i$, with $x_i \in (span\{w_j\})^\perp$. For $g \in \overline{span}\{w_j\}$, we get
\begin{equation*}
\sum_{i \in I}\big|\langle g,y_i \rangle\big|^2=\sum_{i \in I}\big|\langle g,y_i^\prime \rangle\big|^2=\sum_{i \in I}\big|\langle g,v_i \rangle\big|^2=\big\|g\big\|^2,
\end{equation*}
thereby proving that $\{y_i\}_{i \in I}$ is a Parseval frame for $\overline{span}\{w_j\}$.
\qed\end{proof}

\begin{rem}
Under the hypothesis of the above theorem, if $\{y_i\}_{i \in I}$ is a Parseval frame for $\overline{span}\{w_j\}$, (\ref{cond1}) holds and $(span\{w_j\})^\perp$ has a Parseval frame consisting of only finitely many non-zero elements $\{g_i\}_{i=1}^n$, then $w$ is a weak R-dual of $f^*$ with respect to the Parseval frame $u$ and another Parseval frame $v=\{v_i\}_{i \in I}$, given by $v_i=y_i^*+q_i^{**}$, where $q_i=\begin{cases}
g_i, & \text{$1 \leq i \leq n$}\\
0, & \text{$i > n$}
\end{cases}$ \hspace{0.5em}and for a sequence $h=\{h_i\}_{i \in I}$, $h^*$ and $h^{**}$ are the sequences given by $h_i^*
:= \begin{cases}
h_{\frac{i+1}{2}}, & i\ is\ odd,\ 1 \leq i \leq 2n\\
0, & i\ is\ even,\ 1 \leq i \leq 2n\\
h_{\frac{i+k}{2}}, & i=2n+k,\ k \in \mathbb{N}
\end{cases}$ \hspace{0.5em}and \hspace{0.5em}$h_i^{**}
:=
\begin{cases}
0, & i\ is\ odd,\ 1 \leq i \leq 2n\\
h_{\frac{i}{2}}, & i\ is\ even,\ 1 \leq i \leq 2n\\
h_{\frac{i+k}{2}}, & i=2n+k,\ k \in \mathbb{N}
\end{cases}$ \hspace{0.5em} respectively.
\end{rem}

In order to obtain a given frame sequence $w$ as a weak R-dual of a given frame $f$, the following theorem shows that it is enough to look for some weak R-dual of $f$ under certain conditions.
\begin{thm} \label{thm6}
Let $w=\{w_j\}_{j \in I}$ be a frame sequence, $f=\{f_i\}_{i \in I}$ be a frame and $u=\{u_i\}_{i \in I}$ be a Parseval frame for $H$ such that (\ref{cond1}) holds. Suppose the sequence $\{y_i\}_{i \in I}$, given by (\ref{def of ni}), is a Parseval frame for $\overline{span}\{w_j\}_{j \in I}$ and there exists a weak R-dual $\{p_{j}\}_{j \in I}$ of $f$ with respect to $u$ and a Parseval frame $h=\{h_i\}_{i \in I}$ such that $dim((span\{w_{j}\})^\perp)\leq dim((span\{p_{j}\})^\perp)$. Then, there exists a coisometry $U:H \longrightarrow H$ such that $w_j=\sum\limits_{i \in I}\langle f_i,u_j \rangle U(h_i)$. In particular, if the dimensions are equal, then $w$ is a weak R-dual of $f$ with respect to the Parseval frames $u$ and $U(h)$.
\end{thm}
\begin{proof}
Let $\{c_j\}_{j \in I} \in \ell^2(I)$. As $\{y_i\}_{i \in I}$ is a Parseval frame for $\overline{span}\{w_j\}_{j \in I}$, we have
\begin{align*}
\left\|\sum\limits_{m \in I}c_mw_m \right\|^2&=\sum_{i \in I}\left|\left\langle \sum\limits_{m \in I}c_mw_m,y_i \right\rangle\right|^2 \\
&=\sum_{i \in I}\left|\left\langle \sum\limits_{m \in I}c_mw_m,\sum_{k \in I}\big\langle u_k,f_i \big\rangle \widetilde{w}_k \right\rangle\right|^2 \\
&=\sum_{i \in I}\left|\sum_{k \in I}\langle f_i,u_k \rangle \left\langle \sum\limits_{m \in I}c_mw_m,\widetilde{w}_k \right\rangle\right|^2 \\
&=\sum_{i \in I}\left(\sum_{k \in I} \langle f_i,u_k \rangle \left\langle \sum\limits_{m \in I}c_mw_m,\widetilde{w}_k \right\rangle\right)\left(\sum_{l \in I} \langle u_l,f_i \rangle \left\langle \widetilde{w}_l,\sum\limits_{n \in I}c_nw_n \right\rangle\right) \\
&=\sum_{i \in I}\left(\sum\limits_{m \in I}c_m\left(\sum_{k \in I} \langle f_i,u_k \rangle \left\langle w_m,\widetilde{w}_k \right\rangle\right)\right)\left(\sum\limits_{n \in I}\overline{c_n}\left(\sum_{l \in I} \langle u_l,f_i \rangle \left\langle \widetilde{w}_l,w_n \right\rangle\right)\right),
\end{align*}
with the change in the order of summation being possible, by applying Theorem 1 of \cite{swartz1992iterated}. For, taking $\{k_l\}_{l \in I}$ to be an increasing sequence of positive integers, we have
\begin{align*}
\sum\limits_{m \in I}\left|c_m\left(\sum_{l \in I} \langle f_i,u_{k_l} \rangle \langle w_m,\widetilde{w}_{k_l} \rangle\right) \right|&=\sum\limits_{m \in I}\left|c_m \left\langle w_m,\sum_{l \in I} \langle u_{k_l},f_i \rangle\widetilde{w}_{k_l} \right\rangle \right|\\
&\leq \left(\sum_{m \in I}|c_{m}|^2\right)^\frac{1}{2}\left(\sum_{m \in I}\left|\left\langle w_m,\sum_{l \in I} \langle u_{k_l},f_i \rangle\widetilde{w}_{k_l} \right\rangle\right|^2\right)^\frac{1}{2}\\
&\leq \|c\| \sqrt{B_w}\left\|\sum_{l \in I} \langle u_{k_l},f_i \rangle\widetilde{w}_{k_l}\right\|,
\end{align*} where $B_w$ is an upper frame bound of $w$. As $\{\langle u_{k_l},f_i \rangle\}_{l \in I} \in \ell^2(I)$, the right hand side of the above inequality is finite. Now, by (\ref{cond1}), we get
\begin{align*}
\left\|\sum\limits_{m \in I}c_mw_m \right\|^2&=\sum_{i \in I}\left(\sum\limits_{m \in I}c_m\langle f_i,u_m \rangle\right)\left(\sum\limits_{n \in I}\overline{c_n}\langle u_n,f_i \rangle\right) \\
&=\sum_{i \in I}\left \langle f_i,\sum\limits_{m \in I}\overline{c_m}u_m \right\rangle \left \langle \sum\limits_{n \in I}\overline{c_n}u_n,f_i \right\rangle\\
&=\left \langle \sum_{i \in I}\left \langle \sum\limits_{n \in I}\overline{c_n}u_n,f_i \right\rangle f_i,\sum\limits_{m \in I}\overline{c_m}u_m \right\rangle \\
&=\left\langle S_f\left(\sum\limits_{n \in I}\overline{c_n}u_n\right),\sum\limits_{m \in I}\overline{c_m}u_m \right\rangle \\
&=\sum\limits_{m \in I}c_m\left\langle S_f\left(\sum\limits_{n \in I}\overline{c_n}u_n\right),u_m \right\rangle\\
&=\sum\limits_{m \in I}c_m\sum\limits_{n \in I}\overline{c_n}\langle S_f(u_n),u_m \rangle\\
&=\sum\limits_{m \in I}c_m\sum\limits_{n \in I}\overline{c_n} \left\langle \sum_{i \in I}\langle u_n,f_i \rangle f_i,u_m \right\rangle \\
&=\sum\limits_{m \in I}c_m\sum\limits_{n \in I}\overline{c_n}\sum_{i \in I}\langle f_i,u_m \rangle\langle u_n,f_i \rangle \\
&=\sum\limits_{m \in I}c_m\sum\limits_{n \in I}\overline{c_n}\sum_{i \in I}\langle f_i,u_m \rangle \sum_{l \in I}\langle u_n,f_l \rangle \langle h_i,h_l \rangle,
\end{align*}
as $(G(h,h)^t-\mathcal{I})G(f,u)=\bold{0}$ from the definition of the weak R-duality of $\{p_{j}\}_{j \in I}$. Now,
\begin{align}
\left\|\sum\limits_{m \in I}c_mw_m \right\|^2&=\sum\limits_{m \in I}c_m\sum\limits_{n \in I}\overline{c_n}\left\langle \sum_{i \in I}\langle f_i,u_m \rangle h_i, \sum_{l \in I}\langle f_l,u_n \rangle h_l \right\rangle \nonumber\\
&=\sum\limits_{m \in I}c_m\sum\limits_{n \in I}\overline{c_n}\langle p_m,p_n \rangle\nonumber\\
&=\left\langle
\sum\limits_{m \in I}c_mp_m,\sum\limits_{n \in I}c_np_n \right\rangle \nonumber\\
&=\left\|\sum\limits_{m \in I}c_mp_m\right\|^2. \label{eq10}
\end{align}
As it is known that the weak R-dual of a frame is a frame sequence, we note that the sequence $\{p_{j}\}_{j \in I}$ is a frame sequence. Therefore, we may define an operator $U_1:\overline{span}\{p_j\}_{j \in I} \longrightarrow \overline{span}\{w_j\}_{j \in I}$ by $U_1\left(\sum\limits_{j \in I}c_jp_j\right)=\sum\limits_{j \in I}c_jw_j$, which is well-defined by (\ref{eq10}). It can be easily seen that $U_1$ is linear, isometry, onto and hence unitary.
\par
Suppose, $dim((span\{w_{j}\})^\perp)<dim((span\{p_{j}\})^\perp)$. Let $dim((span\{w_{j}\})^\perp)=n<\infty$. We may take $dim(span\{p_{j}\}_{j \in I}^\perp)=d$, which can be finite or infinite. So, we define an operator $U_2:(span\{p_j\})^\perp \longrightarrow (span\{w_j\})^\perp$ by $U_2(\phi_j)=
\begin{cases}
\psi_j, & \text{$1\leq j\leq n$}\\
0, & \text{$j>n$}
\end{cases}$, where $\{\phi_j\}$ and $\{\psi_j\}$ are orthonormal bases of $(span\{p_j\})^\perp$ and $(span\{w_j\})^\perp$ respectively, which is extended linearly to $(span\{p_{j}\}_{j \in I})^\perp$. One can easily see that $U_2$ is bounded. Further, its adjoint operator $U_2^*:(span\{w_{j}\})^\perp \longrightarrow (span\{p_{j}\})^\perp$ is given by $U_2^*\left(\sum\limits_{j=1}^n c_j\psi_j\right)=\sum\limits_{j=1}^nc_j\phi_j$, which is a linear isometry. Now, define an operator $U:=U_1\oplus U_2$ on $H$. In other words, $U(z_1+z_2)=U_1(z_1)+U_2(z_2)$, where $z_1 \in \overline{span}\{p_j\}_{j \in I}$ and $z_2 \in (span\{p_j\}_{j \in I})^\perp$. Then the adjoint operator $U^*$ can be written as $U^*=U_1^*\oplus U_2^*$. It can be easily verified that $U^*$ is an isometry and so $U$ is a coisometry.
\par
On the other hand, in the case when $dim((span\{w_{j}\})^\perp)=dim((span\{p_{j}\})^\perp)$  (finite or infinite), we define an operator $U_2:(span\{p_j\})^\perp \longrightarrow (span\{w_j\})^\perp$ by $U_2(\phi_j)=\psi_j$, where $\{\phi_j\}$ and $\{\psi_j\}$ are orthonormal bases of $(span\{p_j\}_{j \in I})^\perp$ and $(span\{w_j\}_{j \in I})^\perp$ respectively. This operator, extended linearly to $(span\{p_j\})^\perp$, is unitary. Further, define an operator $U:=U_1\oplus U_2$ on $H$, which is a unitary operator.
\par
Thus, in either case, we have a coisometry $U$ on $H$ such that $w_j=U(p_j)=U\left(\sum\limits_{i \in I}\langle f_i,u_j \rangle h_i \right)$ $=\sum\limits_{i \in I}\langle f_i,u_j \rangle U(h_i)$. Further, $U(h)$ is a Parseval frame for $H$, as $\sum\limits_{i \in I}\left|\langle q,U(h_i)\rangle\right|^2=\sum\limits_{i \in I}\left|\langle U^*(q),h_i\rangle\right|^2\\ =\|U^*(q)\|^2=\|q\|^2,\ \forall\ q \in H$. When the dimensions of $(span\{p_{j}\}_{j \in I})^\perp$ and $(span\{w_{j}\}_{j \in I})^\perp$ are equal, $(G(U(h),U(h))^t-\mathcal{I})G(f,u)=(G(h,h)^t-\mathcal{I})G(f,u)=\bold{0}$, thereby proving the theorem.
\qed\end{proof}

The existence of a Parseval frame $u$ satisfying (\ref{cond1}), for a given frame $f$ and frame sequence $w$, is assumed in the hypothesis of the above theorem. The following proposition shows that such a Parseval frame does exist. Moreover, the Parseval frame so constructed is not an orthonormal basis, which is crucial for true weak R-duality.
\begin{pro}
Let $w=\{w_j\}_{j \in I}$ be a frame sequence in $H$ and $f=\{f_i\}_{i \in I}$ be any sequence in $H$. Then, there exists a Parseval frame $u=\{u_i\}_{i \in I}$ (not an orthonormal basis) for $H$ satisfying $(G(\widetilde{w},w)^t-\mathcal{I})G(u,f)=\bold{0}$, where $\widetilde{w}=\{\widetilde{w}_j\}_{j \in I}$ is the canonical dual of $w$.
\end{pro}
\begin{proof}
If $w$ is a Riesz sequence, then the proof is trivial as $(G(\widetilde{w},w)^t-\mathcal{I})=\bold{0}$, by the biorthogonality of $w$ and $\widetilde{w}$. So, we shall assume that $w$ is a frame sequence but not a Riesz sequence. Define an operator $\widetilde{L}:H \longrightarrow \overline{span}\{w_j\}$ by $\widetilde{L}\left(\sum\limits_{i \in I}c_ih_i\right)=\sum\limits_{i \in I}\overline{c_i}g_i$, where $\{h_i\}_{i \in I}$ and $\{g_i\}_{i \in I}$ are orthonormal bases of $H$ and $\overline{span}\{w_j\}$ respectively, and $\{c_i\}_{i \in I} \in \ell^2(I)$. This operator is a conjugate-linear and surjective isometry. For $x \in H$ and $z \in \overline{span}\{w_j:j \in I\}$, we get $\langle \widetilde{L}x,z \rangle=\langle (\widetilde{L})^{-1}z,x \rangle$, following the lines of the proof of Proposition \ref{pro1}. Consequently,
\begin{align*}
\sum_{k \in I}\big\langle \widetilde{w}_k,w_i \big\rangle \big\langle \widetilde{L}^{-1}(S_w^{-1/2}w_k),f_j \big\rangle&=\sum_{k \in I}\big\langle \widetilde{w}_k,w_i \big\rangle \big\langle \widetilde{L}(f_j),S_w^{-1/2}w_k \rangle\\
&=\left\langle\sum_{k \in I}\big\langle \widetilde{L}(f_j),S_w^{-1/2}w_k \rangle \widetilde{w}_k,w_i \right\rangle\\
&=\left\langle\sum_{k \in I}\big\langle S_w^{-1/2}\widetilde{L}(f_j),w_k \rangle \widetilde{w}_k,w_i \right\rangle\\
&=\langle S_w^{-1/2}\widetilde{L}(f_j),w_i \rangle\\
&=\langle\widetilde{L}(f_j),S_w^{-1/2}w_i \rangle\\
&=\langle\widetilde{L}^{-1}(S_w^{-1/2}w_i),f_j \rangle,
\end{align*}
which shows that $(G(\widetilde{w},w)^t-\mathcal{I})G(u,f)=\bold{0}$, wherein $u_i=\widetilde{L}^{-1}(S_w^{-1/2}w_i), i \in I$. As $w$ is not a Riesz sequence, we can prove that $\{\widetilde{L}^{-1}(S_w^{-1/2}w_i)\}_{i \in I}$ is a Parseval frame for $H$ but not an orthonormal basis, as done in Proposition \ref{pro1}.
\qed\end{proof}

The existence of weak R-dual as discussed in Theorems \ref{thm2}, \ref{thm7}, \ref{thm3} and \ref{thm6} involves the condition that $\{y_i\}_{i \in I}$, defined in (\ref{def of ni}), is a Parseval frame for $\overline{span}\{w_j\}$. The following theorem provides a characterization for the sequence $\{y_i\}_{i \in I}$ to be a Parseval frame for $\overline{span}\{w_j\}$.
\begin{thm}\label{thm9}
Let $H$ be a complex separable Hilbert space. Suppose $w=\{w_j\}_{j \in I}$ is a frame sequence and $f=\{f_i\}_{i \in I}$ is a frame for $H$. Then, there exists a Parseval frame $u=\{u_k\}_{k \in I}$ for $H$ satisfying (\ref{cond1}) and the sequence $\{y_i\}_{i \in I}$, given by (\ref{def of ni}), is a Parseval frame for $\overline{span}\{w_j\}$ if and only if there exists a conjugate-linear bounded invertible operator $\widetilde{L}$ from $H$ to $\overline{span}\{w_j\}$ such that the frame operators $S_w$ of $w$ and $S_f$ of $f$ satisfy $S_w=(\widetilde{L}(\widetilde{L})^*)^{-1}$ and $S_f=((\widetilde{L})^*\widetilde{L})^{-1}$ respectively.
\end{thm}
\begin{proof}
Suppose that there exists a Parseval frame $u=\{u_i\}_{i \in I}$ satisfying (\ref{cond1}) and $\{y_i\}_{i \in I}$ is a Parseval frame for $\overline{span}\{w_j\}$. Define a map $\widetilde{L}:H\longrightarrow \overline{span}\{w_k:k \in I\}$ by $\widetilde{L}x=\widetilde{L}\left(\sum\limits_{k \in I}c_ku_k\right)=\sum\limits_{k \in I}\overline{c_k}\widetilde{w}_k$, where $\{c_k\}_{k \in I} \in \ell^2(I)$ and $\widetilde{w}=\{\widetilde{w}_k\}_{k \in I}$ is the canonical dual of $w$. It is easy to see that this map is well-defined. In fact, using (\ref{eq6}), we have $\left\langle \sum\limits_{j \in I}c_jw_j,y_i \right\rangle=\left\langle f_i,\sum\limits_{j \in I}\overline{c_j}u_j \right\rangle$. Denoting $Ker^*(T_u)$ to be the set $\{\{\overline{c_k}\}:\{c_k\} \in Ker(T_u)\}$, we then obtain
\begin{equation}\label{eq9}
Ker^*(T_u)=Ker(T_w)=Ker(T_{\widetilde{w}}).
\end{equation} Clearly, $\widetilde{L}$ is conjugate-linear. Also, $\left\|\widetilde{L}\left(\sum\limits_{k \in I}c_ku_k\right)\right\|=\left\|\sum\limits_{k \in I}\overline{c_k}\widetilde{w}_k\right\|\leq \frac{1}{\sqrt{A_w}}\|c\|=\frac{1}{\sqrt{A_w}}\left\|\sum\limits_{k \in I}c_ku_k\right\|$, where $A_w$ is a lower frame bound of $w$ and so $\widetilde{L}$ is bounded. Now, $\widetilde{L}$ is one-to-one for, if $x=\sum\limits_{k \in I}c_ku_k \in Ker(\widetilde{L})$, then $\{\overline{c_k}\}_{k \in I} \in Ker(T_{\widetilde{w}})$ and hence $\{\overline{c_k}\}_{k \in I} \in Ker^*(T_u)$, using (\ref{eq9}), which gives $x=0$. Obviously, the map $\widetilde{L}$ is onto and hence invertible.\\
\indent For $x \in H$ and $z \in \overline{span}\{w_k:k \in I\}$, we get
\begin{align}
\langle \widetilde{L}x,z \rangle&=\left\langle \widetilde{L}\left(\sum_{k \in I} \langle x,u_k \rangle u_k\right),z \right\rangle
=\left\langle \sum_{k \in I} \langle u_k,x \rangle \widetilde{w}_k,z \right\rangle
=\sum_{k \in I} \langle u_k,x \rangle \langle \widetilde{w}_k,z \rangle \nonumber\\
&=\left\langle \sum_{k \in I} \langle \widetilde{w}_k,z \rangle u_k,x \right\rangle
=\left\langle (\widetilde{L})^{-1}\left(\sum_{k \in I} \langle z,\widetilde{w}_k \rangle \widetilde{w}_k\right),x \right\rangle \nonumber\\
&=\langle (\widetilde{L})^{-1}S_w^{-1}z,x \rangle, \label{eq11}
\end{align}
from which we obtain, $S_w=(\widetilde{L}(\widetilde{L})^*)^{-1}$. Now, for $g \in \overline{span}\{w_j\}$, consider
\begin{align*}
\langle g,g \rangle&=\sum_{i \in I}\left|\langle g,y_i \rangle\right|^2\\
&=\sum_{i \in I}\left|\left\langle g,\sum_{k \in I}\big\langle u_k,f_i \big\rangle \widetilde{w}_k \right\rangle\right|^2\\
&=\sum_{i \in I}\left|\sum_{k \in I}\langle f_i,u_k \rangle \langle g,\widetilde{w}_k \rangle\right|^2\\
&=\sum_{i \in I}\sum_{k \in I} \langle f_i,u_k \rangle \big\langle g,\widetilde{w}_k \big\rangle\sum_{l \in I} \langle u_l,f_i \rangle \big\langle \widetilde{w}_l,g \big\rangle.
\end{align*} In order to change the order of summation above, we let $\{k_j\}_{j \in I}$ to be an increasing sequence of positive integers and consider,
\begin{align*}
\sum_{i \in I}\left|\left\langle f_i,\sum_{j \in I}\langle \widetilde{w}_{k_j},g\rangle u_{k_j}\right\rangle \left\langle \sum_{l \in I}\langle \widetilde{w}_l,g\rangle u_l,f_i \right\rangle\right|&\leq \left(\sum_{i \in I}\left|\left\langle f_i,\sum_{j \in I}\langle \widetilde{w}_{k_j},g\rangle u_{k_j}\right\rangle \right|^2 \right)^\frac{1}{2}\\ &\qquad \quad\left(\sum_{i \in I}\left|\left\langle \sum_{l \in I}\langle \widetilde{w}_l,g\rangle u_l,f_i \right\rangle\right|^2 \right)^\frac{1}{2}\\
&\leq B_f \left\|\sum_{j \in I}\langle \widetilde{w}_{k_j},g\rangle u_{k_j}\right\| \left\|\sum_{l \in I}\langle \widetilde{w}_l,g\rangle u_l\right\|\\
&\leq B_f \left(\sum_{j \in I}|\langle \widetilde{w}_{k_j},g\rangle|^2\right)^\frac{1}{2} \left(\sum_{l \in I}|\langle \widetilde{w}_l,g\rangle|^2\right)^\frac{1}{2}\\
&\leq B_f \frac{1}{A_w} \|g\|^2<\infty,
\end{align*} where $B_f$ and $A_w$ are the upper and lower frame bounds of $f$ and $w$ respectively. By Theorem 1 of \cite{swartz1992iterated}, we get
\begin{align*}
\langle g,g\rangle&=\sum_{k \in I}\left(\sum_{i \in I}\sum_{l \in I}\langle f_i,u_k \rangle \langle u_l,f_i \rangle \big\langle \widetilde{w}_l,g \big\rangle\right)\big\langle g,\widetilde{w}_k \big\rangle.
\end{align*} In a similar manner, we may change the order of summation once again and obtain
\begin{align}
\langle g,g\rangle&=\sum_{k \in I}\sum_{l \in I}\sum_{i \in I} \langle f_i,u_k \rangle \langle u_l,f_i \rangle\big\langle g,\widetilde{w}_k \big\rangle \big\langle \widetilde{w}_l,g \big\rangle \nonumber\\
&=\sum_{k \in I}\sum_{l \in I}\left\langle u_l,\sum_{i \in I}  \langle u_k,f_i \rangle f_i \right\rangle\big\langle g,\widetilde{w}_k \big\rangle \big\langle \widetilde{w}_l,g \big\rangle \nonumber\\
&=\sum_{k \in I}\sum_{l \in I}\langle u_l,S_fu_k\rangle\big\langle g,\widetilde{w}_k \big\rangle \big\langle \widetilde{w}_l,g \big\rangle \nonumber\\
&=\sum_{k \in I}\sum_{l \in I}\langle u_l,S_fu_k\rangle \big\langle g,\widetilde{L}u_k \big\rangle \big\langle \widetilde{L}u_l,g \big\rangle.\label{eq8}
\end{align}
Using (\ref{eq11}) in (\ref{eq8}), we have
\begin{align}
\langle g,g \rangle&=\sum_{k \in I}\sum_{l \in I}\langle u_l,S_fu_k\rangle\big\langle u_k,(\widetilde{L})^{-1}S_w^{-1}g \big\rangle \big\langle (\widetilde{L})^{-1}S_w^{-1}g,u_l \big\rangle \nonumber\\
&=\sum_{k \in I}\left\langle\sum_{l \in I}\big\langle (\widetilde{L})^{-1}S_w^{-1}g,u_l \big\rangle u_l,S_fu_k\right\rangle\big\langle u_k,(\widetilde{L})^{-1}S_w^{-1}g \big\rangle  \nonumber\\
&=\sum_{k \in I}\big\langle (\widetilde{L})^{-1}S_w^{-1}g,S_fu_k \big\rangle\big\langle u_k,(\widetilde{L})^{-1}S_w^{-1}g \big\rangle \nonumber \\
&=\sum_{k \in I}\big\langle S_f (\widetilde{L})^{-1}S_w^{-1}g,u_k \big\rangle \big\langle u_k,(\widetilde{L})^{-1}S_w^{-1}g \big\rangle  \nonumber\\
&=\left\langle S_f (\widetilde{L})^{-1}S_w^{-1}g,\sum_{k \in I}\big\langle(\widetilde{L})^{-1}S_w^{-1}g,u_k \big\rangle u_k\right\rangle \nonumber \\
&=\langle S_f (\widetilde{L})^{-1}S_w^{-1}g,(\widetilde{L})^{-1}S_w^{-1}g\rangle \nonumber \\
&=\langle g,\widetilde{L}S_f (\widetilde{L})^{-1}S_w^{-1}g\rangle. \label{eq12}
\end{align}
The last equality follows from (\ref{eq11}) by taking $x=S_f (\widetilde{L})^{-1}S_w^{-1}g$ and $z=g$. Thus, $\widetilde{L}S_f (\widetilde{L})^{-1}S_w^{-1}\\ =I_d$ and hence $S_f=((\widetilde{L})^*\widetilde{L})^{-1}$.
\par
Conversely, let us suppose that there exists a conjugate-linear bounded invertible operator $\widetilde{L}$ such that $S_w=(\widetilde{L}(\widetilde{L})^*)^{-1}$ and $S_f=((\widetilde{L})^*\widetilde{L})^{-1}$. Define $u_k:=(\widetilde{L})^{-1}\widetilde{w}_k$. Then $\{u_k\}_{k \in I}$ is a frame for $H$. Moreover,
\begin{equation*}
\sum_{k \in I}\langle h,u_k \rangle u_k=\sum_{k \in I}\langle h,(\widetilde{L})^{-1}\widetilde{w}_k \rangle (\widetilde{L})^{-1}\widetilde{w}_k=(\widetilde{L})^{-1}\left(\sum_{k \in I}\langle \widetilde{L}h,w_k \rangle \widetilde{w}_k\right)=(\widetilde{L})^{-1}\widetilde{L}h=h,\ \forall\ h \in H,
\end{equation*}
which implies that $\{u_k\}_{k \in I}$ is a Parseval frame for $H$. Further,
\begin{align*}
\sum_{k \in I}&\langle \widetilde{w}_k,w_j \rangle \langle u_k,f_i \rangle=\sum_{k \in I}\langle \widetilde{w}_k,w_j \rangle \langle (\widetilde{L})^{-1}\widetilde{w}_k,f_i \rangle
=\left\langle\sum_{k \in I}\langle (\widetilde{L})^{-1}\widetilde{w}_k,f_i \rangle \widetilde{w}_k,w_j \right\rangle\\
&=\left\langle\sum_{k \in I}\langle \widetilde{L}f_i,w_k \rangle \widetilde{w}_k,w_j \right\rangle
=\langle \widetilde{L}f_i,w_j \rangle
=\langle (\widetilde{L})^{-1}\widetilde{w}_j,f_i \rangle
=\langle u_j,f_i \rangle,
\end{align*}
which gives $(G(\widetilde{w},w)^t-\mathcal{I})G(u,f)=\bold{0}$. Using the steps in the proofs of (\ref{eq8}) and (\ref{eq12}) as well as the relations $S_w=(\widetilde{L}(\widetilde{L})^*)^{-1}$ and $S_f=((\widetilde{L})^*\widetilde{L})^{-1}$, we obtain for $g \in \overline{span}\{w_j\}_{j \in I}$,
\begin{equation*}
\sum_{i \in I}\left|\langle g,y_i \rangle\right|^2=\langle g,\widetilde{L}S_f (\widetilde{L})^{-1}S_w^{-1}g\rangle=\langle g,g \rangle,
\end{equation*}
which shows that $\{y_i\}_{i \in I}$ is a Parseval frame for $\overline{span}\{w_j\}_{j \in I}$.
\qed\end{proof}

\section{On the characterizing sequence `y'}
We have seen in the previous section that the sequence $y=\{y_i\}_{i \in I}$, defined by (\ref{def of ni}), plays a substantial role in characterizing weak R-duality. We have also given a characterization for $y$ to be a Parseval frame for $\overline{span}\{w_i\}_{i \in I}$ in Theorem \ref{thm9}. In this section, we analyse a few aspects related to this sequence and also show that a condition which arises in this connection turns out to be necessary for weak R-duality. From the definition of the sequence $\{y_i\}_{i \in I}$, it is clear that $\overline{span}\{y_i:i \in I\}\subset \overline{span}\{w_i:i \in I\}$. We remark that in general, the sequence $\{y_i\}_{i \in I}$ need not be complete in $\overline{span}\{w_i\}_{i \in I}$, even if $u$ is a Parseval frame for $H$. An example which shows this is as follows.
\begin{ex}\label{ex1}Let $\{z_i\}_{i \in \mathbb{N}}$ be an orthonormal basis for $H$ and
\begin{align*}
\{u_i\}_{i \in \mathbb{N}}&=\left\{\frac{z_1}{\sqrt{2}},\frac{z_1}{\sqrt{2}},\frac{z_2}{\sqrt{2}},\frac{z_2}{\sqrt{2}},...\right\},\\
\{f_i\}_{i \in \mathbb{N}}&=\{z_1,z_1,z_2,z_2,...\},\\
\{w_i\}_{i \in \mathbb{N}}&=\{z_1,z_3,z_5,z_7,...\}.
\end{align*}
Clearly $\{u_i\}_{i \in \mathbb{N}}$ is a Parseval frame for $H$, $\{w_i\}_{i \in \mathbb{N}}$ is a frame sequence in $H$ and its canonical dual is itself. Using the definition of $y_i$, we get
\begin{align*}
y_{2i}&=\sum_{k \in \mathbb{N}}\big\langle u_{2k-1},f_{2i} \big\rangle w_{2k-1}+\sum_{k \in \mathbb{N}}\big\langle u_{2k},f_{2i} \big\rangle w_{2k}\\
&=\sum_{k \in \mathbb{N}}\left\langle \frac{z_k}{\sqrt{2}},z_i \right\rangle z_{4k-3}+\sum_{k \in \mathbb{N}}\left\langle \frac{z_k}{\sqrt{2}},z_i \right\rangle z_{4k-1}\\
&=\frac{z_{4i-3}}{\sqrt{2}}+\frac{z_{4i-1}}{\sqrt{2}},\ i \in \mathbb{N}.
\end{align*}
Similarly, $y_{2i-1}=\frac{z_{4i-3}}{\sqrt{2}}+\frac{z_{4i-1}}{\sqrt{2}},\ i \in \mathbb{N}.$ So,
\begin{equation*}
\{y_i\}_{i \in \mathbb{N}}=\left\{\frac{z_1+z_3}{\sqrt{2}},\frac{z_1+z_3}{\sqrt{2}},\frac{z_5+z_7}{\sqrt{2}},\frac{z_5+z_7}{\sqrt{2}},...\right\}.
\end{equation*}
Clearly $z_1 \notin span\{y_i\}$. Also $z_1 \notin \overline{span}\{y_i\}$ as $(z_1-z_3)\perp y_i,\ \forall\ i \in \mathbb{N}$ but $(z_1-z_3)\not\perp z_1$. This shows that $\overline{span}\{y_i:i \in \mathbb{N}\}\subsetneq \overline{span}\{w_i:i \in \mathbb{N}\}$.
\end{ex}

The following theorem provides a condition which guarantees the completeness of $\{y_i\}_{i \in I}$ in $\overline{span}\{w_j\}$.
\begin{thm}\label{lem2} Let $w=\{w_i\}_{i \in I}$ be a frame sequence in $H$ and $f=\{f_i\}_{i \in I}$ be a frame for $H$. Let $u=\{u_i\}_{i \in I}$ be a Parseval frame for $H$ such that
\begin{equation}\label{cond2}
(G(u,u)^t-\mathcal{I})w=\bold{0}
\end{equation}
and $\{y_i\}_{i \in I}$ be as defined in (\ref{def of ni}). Then, $\overline{span}\{y_i:i \in I\}=\overline{span}\{w_i:i \in I\}$. In fact, $\{y_i\}_{i \in I}$ is a frame for $\overline{span}\{w_i\}$. If $A_w, B_w$ are frame bounds of $w$ and $A_f, B_f$ are those of $f$, then $\frac{A_f}{B_w},\frac{B_f}{A_w}$ are frame bounds of $\{y_i\}_{i \in I}$.
\end{thm}
\begin{proof} Let $g \in \overline{span}\{w_i\}.$
Consider
\begin{align}
\sum_{i \in I}\left|\langle g,y_i \rangle\right|^2&=\sum_{i \in I}\left|\left\langle g,\sum_{k \in I}\big\langle u_k,f_i \big\rangle \widetilde{w}_k \right\rangle\right|^2 \nonumber\\
&=\sum_{i \in I}\left|\sum_{k \in I}\langle f_i,u_k \rangle \langle g,\widetilde{w}_k \rangle\right|^2\nonumber\\
&=\sum_{i \in I}\left|\left\langle f_i,\sum_{k \in I}\big\langle \widetilde{w}_k,g \big\rangle u_k \right\rangle\right|^2\nonumber\\
&\geq A_f \left\|\sum_{k \in I}\big\langle \widetilde{w}_k,g \big\rangle u_k\right\|^2, \label{eq5}
\end{align}
where $\{\widetilde{w}_k\}_{k \in I}$ is the canonical dual of $w$. Now, we shall show that $\big\{\big \langle \widetilde{w}_k,g \big \rangle\big\}_{k \in I} \in Ker(T_u)^\perp$. Let $\{c_k\}_{k \in I} \in Ker(T_u)$. From (\ref{cond2}) we have, $\widetilde{w}_k=\sum\limits_{m \in I}\big \langle u_m,u_k \big \rangle \widetilde{w}_m$, which gives
\begin{equation*}
\sum\limits_{k \in I}\overline{c_k}\widetilde{w}_k=\sum\limits_{k \in I}\overline{c_k}\sum\limits_{m \in I}\big \langle u_m,u_k \big \rangle \widetilde{w}_m=\sum\limits_{m \in I}\sum\limits_{k \in I}\big \langle u_m,c_ku_k \big \rangle \widetilde{w}_m.
\end{equation*} Here also the change in the order of summation is possible, by applying Theorem 1 of \cite{swartz1992iterated}, as done in the proof of Theorem \ref{thm4}. Thus, $\sum\limits_{k \in I}\overline{c_k}\widetilde{w}_k=\sum\limits_{m \in I}\big \langle u_m,\sum\limits_{k \in I}c_ku_k \big \rangle \widetilde{w}_m=0$, as $\{c_k\}_{k \in I} \in Ker(T_u)$. This in turn implies that $\sum\limits_{k \in I}\big\langle \widetilde{w}_k,g \big\rangle \overline{c_k}=\left\langle \sum\limits_{k \in I}\overline{c_k}\widetilde{w}_k,g \right\rangle=0$, thereby proving that $\big\{\big \langle \widetilde{w}_k,g \big \rangle\big\}_{k \in I} \in Ker(T_u)^\perp$. From (\ref{eq5}), it then follows that
\begin{equation*}
\sum_{i \in I}\big|\langle g,y_i \rangle\big|^2\geq A_f \sum_{k \in I}\big|\big\langle \widetilde{w}_k,g \big\rangle\big|^2\geq \frac{A_f}{B_w} \big\|g\big\|^2.
\end{equation*}
The upper frame inequality with bound $\frac{B_f}{A_w}$ can be proved similarly.
\qed\end{proof}

The converse of the above theorem is not true. More explicitly, if $\{y_i\}_{i \in I}$ is a frame for $\overline{span}\{w_i\}$, then (\ref{cond2}) need not be satisfied, as can be seen from the example below.
\begin{ex}
Let $\{z_i\}_{i \in \mathbb{N}}$ be an orthonormal basis for $H$,
\begin{align*}
\{f_i\}_{i \in \mathbb{N}}&=\{z_1,z_2,z_3,z_4,...\},\\
\{u_i\}_{i \in \mathbb{N}}&=\left\{\frac{z_1}{\sqrt{2}},\frac{z_1}{\sqrt{2}},\frac{z_2}{\sqrt{2}},\frac{z_2}{\sqrt{2}},...\right\}\\
and \hspace{0.5em} \{w_i\}_{i \in \mathbb{N}}&=\left\{\frac{z_2}{\sqrt{2}},\frac{z_3}{\sqrt{2}},\frac{z_2}{\sqrt{2}},\frac{-z_3}{\sqrt{2}},\frac{z_4}{\sqrt{2}},\frac{z_5}{\sqrt{2}},\frac{z_4}{\sqrt{2}},\frac{-z_5}{\sqrt{2}},...\right\}.
\end{align*}
Clearly, $\{u_i\}_{i \in \mathbb{N}}$ is a Parseval frame for $H$, $\{w_i\}_{i \in \mathbb{N}}$ is a frame sequence in $H$ and its canonical dual is itself. For the ease of computation, we shall rewrite the sequences $\{u_i\}_{i \in \mathbb{N}}$ and $\{w_i\}_{i \in \mathbb{N}}$ as
\begin{equation*}
u_i=\begin{cases}
\frac{z_{2k-1}}{\sqrt{2}}, & i=4k-3,\ k \in \mathbb{N} \\
\frac{z_{2k-1}}{\sqrt{2}}, & i=4k-2,\ k \in \mathbb{N}\\
\frac{z_{2k}}{\sqrt{2}}, & i=4k-1,\ k \in \mathbb{N} \\
\frac{z_{2k}}{\sqrt{2}}, & i=4k,\ k \in \mathbb{N}
\end{cases} \hspace{1em} and \hspace{1em} w_i=
\begin{cases}
\frac{z_{2k}}{\sqrt{2}}, & i=4k-3,\ k \in \mathbb{N} \\
\frac{z_{2k+1}}{\sqrt{2}}, & i=4k-2,\ k \in \mathbb{N} \\
\frac{z_{2k}}{\sqrt{2}}, & i=4k-1,\ k \in \mathbb{N} \\
\frac{-z_{2k+1}}{\sqrt{2}}, & i=4k,\ k \in \mathbb{N}.
\end{cases}
\end{equation*}
Using the definition of $y_i$, we get
\begin{align*}
y_{2i}&=\sum_{k \in \mathbb{N}}\big\langle u_{4k-3},f_{2i} \big\rangle w_{4k-3}+\sum_{k \in \mathbb{N}}\big\langle u_{4k-2},f_{2i} \big\rangle w_{4k-2}+\sum_{k \in \mathbb{N}}\big\langle u_{4k-1},f_{2i} \big\rangle w_{4k-1}+\sum_{k \in \mathbb{N}}\big\langle u_{4k},f_{2i} \big\rangle w_{4k}\\
&=\sum_{k \in \mathbb{N}}\left\langle \frac{z_{2k-1}}{\sqrt{2}},z_{2i} \right\rangle \frac{z_{2k}}{\sqrt{2}}+\sum_{k \in \mathbb{N}}\left\langle \frac{z_{2k-1}}{\sqrt{2}},z_{2i} \right\rangle \frac{z_{2k+1}}{\sqrt{2}}+\sum_{k \in \mathbb{N}}\left\langle \frac{z_{2k}}{\sqrt{2}},z_{2i} \right\rangle \frac{z_{2k}}{\sqrt{2}}\\ &\qquad+\sum_{k \in \mathbb{N}}\left\langle \frac{z_{2k}}{\sqrt{2}},z_{2i} \right\rangle \frac{-z_{2k+1}}{\sqrt{2}}\\
&=\frac{z_{2i}}{2}-\frac{z_{2i+1}}{2},\ i \in \mathbb{N}.
\end{align*}
Similarly, $y_{2i-1}=\frac{z_{2i}}{2}+\frac{z_{2i+1}}{2},\ i \in \mathbb{N}.$ So,
\begin{equation*}
\{y_i\}_{i \in \mathbb{N}}=\left\{\frac{z_2+z_3}{2},\frac{z_2-z_3}{2},\frac{z_4+z_5}{2},\frac{z_4-z_5}{2},...\right\},
\end{equation*} which is clearly a frame for $\overline{span}\{w_i\}$ with both the frame bounds being $\frac{1}{2}$.
\par Now, consider
\begin{align*}
\sum_{k \in \mathbb{N}}\langle u_k,u_1 \rangle w_k&=\sum_{k \in \mathbb{N}}\langle u_{4k-3},u_1 \rangle w_{4k-3}+\sum_{k \in \mathbb{N}}\langle u_{4k-2},u_1 \rangle w_{4k-2}+\sum_{k \in \mathbb{N}}\langle u_{4k-1},u_1 \rangle w_{4k-1}\\ &\qquad+\sum_{k \in \mathbb{N}}\langle u_{4k},u_1 \rangle w_{4k}\\
&=\sum_{k \in \mathbb{N}}\left\langle \frac{z_{2k-1}}{\sqrt{2}},\frac{z_1}{\sqrt{2}} \right\rangle \frac{z_{2k}}{\sqrt{2}}+\sum_{k \in \mathbb{N}}\left\langle \frac{z_{2k-1}}{\sqrt{2}},\frac{z_1}{\sqrt{2}} \right\rangle \frac{z_{2k+1}}{\sqrt{2}}+\sum_{k \in \mathbb{N}}\left\langle \frac{z_{2k}}{\sqrt{2}},\frac{z_1}{\sqrt{2}} \right\rangle \frac{z_{2k}}{\sqrt{2}}\\ &\qquad+\sum_{k \in \mathbb{N}}\left\langle \frac{z_{2k}}{\sqrt{2}},\frac{z_1}{\sqrt{2}} \right\rangle \frac{-z_{2k+1}}{\sqrt{2}}\\
&=\frac{z_2}{2\sqrt{2}}+\frac{z_3}{2\sqrt{2}}\\
&\neq w_1,
\end{align*}
which implies that $(G(u,u)^t-\mathcal{I})w\neq \bold{0}$.
\end{ex}

However, the converse of Theorem \ref{lem2} is true when an additional condition is imposed, which is proved as follows.
\begin{thm}
Let $w=\{w_j\}_{j \in I}$ be a frame sequence in $H$, $\widetilde{w}=\{\widetilde{w}_j\}_{j \in I}$ be its canonical dual, $u=\{u_i\}_{i \in I}$ be a Parseval frame for $H$ and $f=\{f_i\}_{i \in I}$ be any sequence in $H$ such that the sequence $\{y_i\}_{i \in I}$, given by (\ref{def of ni}), is complete in $\overline{span}\{w_j\}_{j \in I}$. In addition, if $(G(\widetilde{w},w)^t-\mathcal{I})G(u,f)=\bold{0}$ holds, then $(G(u,u)^t-\mathcal{I})w=\bold{0}$.
\end{thm}
\begin{proof}
If $(G(\widetilde{w},w)^t-\mathcal{I})G(u,f)=\bold{0}$, then $\langle w_j,y_i\rangle=\left\langle w_j,\sum\limits_{k \in I}\big\langle u_k,f_i \big\rangle \widetilde{w}_k \right\rangle=\overline{\sum\limits_{k \in I}\big\langle u_k,f_i \big\rangle \langle\widetilde{w}_k,w_j \rangle}\\ =\langle f_i,u_j\rangle$. So, for $k \in I$,
\begin{align*}
\left \langle\sum\limits_{j \in I}\langle u_j,u_k \rangle w_j,y_i \right\rangle&=\sum\limits_{j \in I}\langle u_j,u_k \rangle \langle w_j,y_i \rangle=\sum\limits_{j \in I}\langle u_j,u_k \rangle \langle f_i,u_j \rangle\\
&=\left\langle \sum\limits_{j \in I}\langle f_i,u_j \rangle u_j,u_k \right\rangle=\langle f_i,u_k \rangle\\
&=\langle w_k,y_i \rangle,
\end{align*}
which in turn gives $\left \langle \sum\limits_{j \in I}\langle u_j,u_k \rangle w_j-w_k,g \right\rangle=0,\ \forall\ g \in \overline{span}\{y_i\}$. As $\{y_i\}_{i \in I}$ is complete in $\overline{span}\{w_i\}$, we get $\sum\limits_{j \in I}\langle u_j,u_k \rangle w_j-w_k=0,\ \forall\ k \in I$, which is equivalent to $(G(u,u)^t-\mathcal{I})w=\bold{0}$.
\qed\end{proof}

Furthermore, it is true that (\ref{cond2}) is a necessary condition that holds when $w$ is a weak R-dual of $f$. Nevertheless, we shall state and prove a more general result as follows.
\begin{thm}
Let $u=\{u_i\}_{i \in I}$ be a Parseval frame, $v=\{v_i\}_{i \in I}, f=\{f_i\}_{i \in I}$ be Bessel sequences and $w=\{w_i\}_{i \in I}$ be any sequence in $H$ such that $w_j=\sum\limits_{i \in I}\langle f_i,u_j \rangle v_i, j \in I$. Then, $(G(u,u)^t-\mathcal{I})w=\bold{0}$.
\end{thm}
\begin{proof}
For $j \in I$,
\begin{equation}\label{eq7}
w_j=\sum_{i \in I}\big\langle f_i,u_j \big\rangle v_i=\sum\limits_{i \in I}\left\langle\sum\limits_{k \in I}\big \langle f_i,u_k \big \rangle u_k,u_j \right\rangle v_i=\sum\limits_{i \in I}\sum\limits_{k \in I}\langle f_i,u_k \rangle \langle u_k,u_j\rangle v_i.
\end{equation}
In order to change the order of summation, we consider an increasing sequence $\{k_l\}_{l \in I}$ of positive integers and we have,
\begin{align*}
\left(\sum_{i \in I}\left|\sum\limits_{l \in I}\langle f_i,u_{k_l} \rangle \langle u_{k_l},u_j\rangle\right|^2 \right)^{\frac{1}{2}}&=\left(\sum_{i \in I}\left|\left\langle f_i,\sum_{l \in I}\langle u_j,u_{k_l} \rangle u_{k_l} \right\rangle\right|^2 \right)^{\frac{1}{2}}\\
&\leq \sqrt{B_f}\left\|\sum_{l \in I}\langle u_j,u_{k_l} \rangle u_{k_l}\right\|<\infty,
\end{align*}
as $\{\langle u_j,u_{k_l}\rangle\}_{l \in I} \in \ell^2(I)$. Here, $B_f$ denotes the Bessel bound of $f$. By applying Theorem 1 of \cite{swartz1992iterated} to the series in (\ref{eq7}), we obtain
\begin{equation*}
w_j=\sum\limits_{k \in I}\sum\limits_{i \in I}\langle f_i,u_k \rangle \langle u_k,u_j\rangle v_i=\sum\limits_{k \in I}\langle u_k,u_j \rangle\left(\sum\limits_{i \in I}\langle f_i,u_k \rangle v_i\right)=\sum\limits_{k \in I}\langle u_k,u_j \rangle w_k,
\end{equation*} which implies that $(G(u,u)^t-\mathcal{I})w=\bold{0}$.
\qed\end{proof}

\section{Weak R-duality of Gabor frames}
In this section, we show that in order to have the adjoint Gabor system as an R-dual of a given Gabor frame, it is enough to check if it is a weak R-dual of the same. Towards this end, we first discern when a weak R-dual gives rise to an R-dual.
\begin{thm}\label{thm8}
Let $w=\{w_j\}_{j \in I}$ be a Riesz sequence and $f=\{f_i\}_{i \in I}$ be a frame for $H$. Suppose $w$ is a weak R-dual of $f$ with respect to Parseval frames $u=\{u_i\}_{i \in I}$ and $v=\{v_i\}_{i \in I}$. If $dim(Ker(T_y))=dim((span\{w_j\})^\perp)$, where $y$ is as defined in (\ref{def of ni}), then $u$ is an orthonormal basis and there exists another orthonormal basis $v^\prime$ such that $w$ is an R-dual of $f$ with respect to $u$ and $v^\prime$.
\end{thm}
\begin{proof}
Under the above hypothesis, it is proved in \cite{li2020class} that $u$ will turn out to be an orthonormal basis. Now, by Theorem \ref{thm1}, (\ref{cond1}) is true and we have $v_i=y_i+x_i$ with $x_i \in (span\{w_j\})^\perp$. For $g \in \overline{span}\{w_j\}$, we have $\sum\limits_{i \in I}\big|\langle g,y_i \rangle\big|^2=\sum\limits_{i \in I}\big|\langle g,v_i \rangle\big|^2=\big\|g\big\|^2$, thereby showing that $y$ is a Parseval frame for $\overline{span}\{w_j\}$. Therefore, by Theorem \ref{thm2}, we may conclude that there exists an orthonormal basis $v^\prime$ of $H$ such that $w$ is an R-dual of $f$ with respect to $u$ and $v^\prime$.
\qed\end{proof}

\begin{cor}\label{cor}
Let $f \in L^2(\mathbb{R})$ and the Gabor system $\{E_{mb}\mathcal{T}_{na}f\}_{m,n \in \mathbb{Z}}, a,b>0$ be a frame for $L^2(\mathbb{R})$. If the adjoint system $\left\{\frac{1}{\sqrt{ab}}E_{\frac{m}{a}}\mathcal{T}_{\frac{n}{b}}f\right\}_{m, n \in \mathbb{Z}}$ is a weak R-dual of $\{E_{mb}\mathcal{T}_{na}f\}_{m,n \in \mathbb{Z}}$, then the adjoint system is also an R-dual of the given Gabor frame.
\end{cor}
\begin{proof}
Suppose $\left\{\frac{1}{\sqrt{ab}}E_{\frac{m}{a}}\mathcal{T}_{\frac{n}{b}}f\right\}_{m, n \in \mathbb{Z}}$ is a weak R-dual of $\{E_{mb}\mathcal{T}_{na}f\}_{m,n \in \mathbb{Z}}$ with respect to two Parseval frames $\{u_{mn}\}_{m,n \in \mathbb{Z}}$ and $\{v_{mn}\}_{m,n \in \mathbb{Z}}$. By the Duality principle for Gabor systems, $\left\{\frac{1}{\sqrt{ab}}E_{\frac{m}{a}}\mathcal{T}_{\frac{n}{b}}f\right\}_{m, n \in \mathbb{Z}}$ is a Riesz sequence. Consider the sequence $y=\{y_{mn}\}_{m,n \in \mathbb{Z}}$ given by
\begin{equation*}
y_{mn}:=\sum\limits_{m^\prime,n^\prime \in \mathbb{Z}}\left\langle u_{m^\prime n^\prime},E_{mb}\mathcal{T}_{na}f \right\rangle S_w^{-1}\left(\frac{1}{\sqrt{ab}}E_{\frac{m^\prime}{a}}\mathcal{T}_{\frac{n^\prime}{b}}f\right),
\end{equation*}
where $S_w$ denotes the frame operator of $\left\{\frac{1}{\sqrt{ab}}E_{\frac{m}{a}}\mathcal{T}_{\frac{n}{b}}f\right\}_{m, n \in \mathbb{Z}}$. In order to prove the corollary, it suffices to show that $dim\left(\left(span\left\{\frac{1}{\sqrt{ab}}E_{\frac{m}{a}}\mathcal{T}_{\frac{n}{b}}f\right\}\right)^\perp\right)=dim(Ker(T_y))$, in view of Theorem \ref{thm8}. First, we observe that $dim(Ker(T_f))=dim(Ker(T_y))$, where $T_f$ denotes the synthesis operator of $\{E_{mb}\mathcal{T}_{na}f\}$. In fact, for $\{c_{mn}\}_{m,n \in \mathbb{Z}} \in \ell^2(\mathbb{Z}^2)$, by (\ref{eq4}), we have
\begin{equation*}
\sum\limits_{{m,n \in \mathbb{Z}}}c_{mn}y_{mn}=\sum\limits_{m^\prime,n^\prime \in \mathbb{Z}}\left\langle u_{m^\prime n^\prime},\sum\limits_{{m,n \in \mathbb{Z}}}\overline{c_{mn}}E_{mb}\mathcal{T}_{na}f \right\rangle S_w^{-1}\left(\frac{1}{\sqrt{ab}}E_{\frac{m^\prime}{a}}\mathcal{T}_{\frac{n^\prime}{b}}f\right).
\end{equation*}
This implies that $\{c_{mn}\}_{m,n \in \mathbb{Z}} \in Ker(T_y)$ if and only if $\{\overline{c_{mn}}\}_{m,n \in \mathbb{Z}} \in Ker(T_f)$, as $\left\{\frac{1}{\sqrt{ab}}E_{\frac{m}{a}}\mathcal{T}_{\frac{n}{b}}f\right\}$ is a Riesz sequence. Thus, we shall prove that $dim\left(\left(span\left\{\frac{1}{\sqrt{ab}}E_{\frac{m}{a}}\mathcal{T}_{\frac{n}{b}}f\right\}\right)^\perp\right)=dim(Ker(T_f))$. It is well known that if $\{E_{mb}\mathcal{T}_{na}f\}_{m,n \in \mathbb{Z}}$ is a frame for $L^2(\mathbb{R})$, then $ab\leq 1$. Further, the Gabor frame $\{E_{mb}\mathcal{T}_{na}f\}_{m,n \in \mathbb{Z}}$ is a Riesz basis if and only if $ab=1$.
\par Considering the case when $ab=1$, we infer that the Gabor frame $\{E_{mb}\mathcal{T}_{na}f\}_{m,n \in \mathbb{Z}}$ is a Riesz basis and hence $dim(Ker(T_f))=0$. Also, $dim\left(\left(span\left\{\frac{1}{\sqrt{ab}}E_{\frac{m}{a}}\mathcal{T}_{\frac{n}{b}}f\right\}\right)^\perp\right)=0$, as $\left\{E_{\frac{m}{a}}\mathcal{T}_{\frac{n}{b}}f\right\}$ is the same as $\{E_{mb}\mathcal{T}_{na}f\}$ in this case. Thus, $dim\left(\left(span\left\{\frac{1}{\sqrt{ab}}E_{\frac{m}{a}}\mathcal{T}_{\frac{n}{b}}f\right\}\right)^\perp\right)=dim(Ker(T_f))$, when $ab=1$. If $ab<1$, it is known that $\left\{\frac{1}{\sqrt{ab}}E_{\frac{m}{a}}\mathcal{T}_{\frac{n}{b}}f\right\}_{m,n \in \mathbb{Z}}$ has infinite deficit. In other words, $dim\left(\left(span\left\{\frac{1}{\sqrt{ab}}E_{\frac{m}{a}}\mathcal{T}_{\frac{n}{b}}f\right\}\right)^\perp\right)=\infty$. Further, by \cite{balan2003deficits}, $dim(Ker(T_f))=0$ or $\infty$. However, $dim(Ker(T_f))$ cannot be $0$, for otherwise, $\{E_{mb}\mathcal{T}_{na}f\}_{m,n \in \mathbb{Z}}$ will turn out to be a Riesz basis, which in turn implies that $ab=1$. Therefore, $dim(Ker(T_f))=\infty=dim\left(\left(span\left\{\frac{1}{\sqrt{ab}}E_{\frac{m}{a}}\mathcal{T}_{\frac{n}{b}}f\right\}\right)^\perp\right)$, thereby proving our assertion.
\qed\end{proof}

We shall now turn to the question as to when the adjoint Gabor system will be a weak R-dual of a Gabor frame. Suppose the given Gabor frame is a Riesz basis. Then, it is proved in \cite{casazza2004duality} that the adjoint Gabor system, which coincides with the given Gabor frame, is an R-dual. Now, if the adjoint Gabor system is a weak R-dual with respect to Parseval frames, say $u$ and $v$, then by Remark \ref{rem}, both $u$ and $v$ will be orthonormal bases indeed. Thus, we have the following.
\begin{thm}
For $f \in L^2(\mathbb{R})$, let $\{E_{mb}\mathcal{T}_{na}f\}_{m,n \in \mathbb{Z}}$ be a Riesz basis for $L^2(\mathbb{R})$. Then, the adjoint Gabor system $\left\{\frac{1}{\sqrt{ab}}E_{\frac{m}{a}}\mathcal{T}_{\frac{n}{b}}f\right\}_{m, n \in \mathbb{Z}}$ is a weak R-dual of $\{E_{mb}\mathcal{T}_{na}f\}_{m,n \in \mathbb{Z}}$ only in the sense of an R-dual.
\end{thm}

On the contrary, if $\{E_{mb}\mathcal{T}_{na}f\}$ is not a Riesz basis but a tight frame, then $\left\{\frac{1}{\sqrt{ab}}E_{\frac{m}{a}}\mathcal{T}_{\frac{n}{b}}f\right\}$, besides being an R-dual (proved in \cite{casazza2004duality}), will also be a weak R-dual in the true sense, which is the essence of the theorem below.
\begin{thm}
Let $f \in L^2(\mathbb{R})$ and the Gabor system $\{E_{mb}\mathcal{T}_{na}f\}_{m,n \in \mathbb{Z}},\ a,b>0$ and $ab<1$, be a tight frame for $L^2(\mathbb{R})$. Then, the adjoint system $\left\{\frac{1}{\sqrt{ab}}E_{\frac{m}{a}}\mathcal{T}_{\frac{n}{b}}f\right\}_{m, n \in \mathbb{Z}}$ is a weak R-dual of $\{E_{mb}\mathcal{T}_{na}f\}_{m,n \in \mathbb{Z}}$ with respect to an orthonormal basis and a Parseval frame, which is not an orthonormal basis, for $L^2(\mathbb{R})$.
\end{thm}
\begin{proof}
We know that the adjoint system $\left\{\frac{1}{\sqrt{ab}}E_{\frac{m}{a}}\mathcal{T}_{\frac{n}{b}}f\right\}_{m, n \in \mathbb{Z}}$ is a Riesz sequence with bounds same as those of $\{E_{mb}\mathcal{T}_{na}f\}_{m,n \in \mathbb{Z}}$. Let $u=\{u_{mn}\}_{m,n \in \mathbb{Z}}$ be any orthonormal basis of $L^2(\mathbb{R})$. By Theorem \ref{lem2}, the sequence $\{y_{mn}\}$ given by $y_{mn}=\sum\limits_{m^\prime,n^\prime \in \mathbb{Z}}\left\langle u_{m^\prime n^\prime},E_{mb}\mathcal{T}_{na}f \right\rangle S_w^{-1}\left(\frac{1}{\sqrt{ab}}E_{\frac{m^\prime}{a}}\mathcal{T}_{\frac{n^\prime}{b}}f\right)$, where $S_w$ is the frame operator of $\left\{\frac{1}{\sqrt{ab}}E_{\frac{m}{a}}\mathcal{T}_{\frac{n}{b}}f\right\}$, is a Parseval frame for $\overline{span}\left\{\frac{1}{\sqrt{ab}}E_{\frac{m}{a}}\mathcal{T}_{\frac{n}{b}}f\right\}$. As $\left\{\frac{1}{\sqrt{ab}}E_{\frac{m}{a}}\mathcal{T}_{\frac{n}{b}}f\right\}_{m, n \in \mathbb{Z}}$ is a Riesz sequence, (\ref{cond1}) is satisfied. Further, for $ab<1$, we have proved in Corollary \ref{cor} that $dim\left(\left(span\left\{\frac{1}{\sqrt{ab}}E_{\frac{m}{a}}\mathcal{T}_{\frac{n}{b}}f\right\}\right)^\perp\right)=dim(Ker(T_y))=\infty$. Therefore, by Theorem \ref{thm7}, there exists some Parseval frame $v=\{v_{mn}\}_{m,n \in \mathbb{Z}}$ for $L^2(\mathbb{R})$, which is not an orthonormal basis, such that the adjoint system is a weak R-dual of the given tight Gabor frame with respect to $u$ and $v$.
\qed\end{proof}

In view of the above discussion, we ask the following interesting question.\\

\noindent\textbf{Open problem:}  Given a Gabor frame $\{E_{mb}\mathcal{T}_{na}f\}_{m,n \in \mathbb{Z}}$, which is neither a Riesz basis nor a tight frame, is the adjoint Gabor system $\left\{\frac{1}{\sqrt{ab}}E_{\frac{m}{a}}\mathcal{T}_{\frac{n}{b}}f\right\}_{m, n \in \mathbb{Z}}$ a weak R-dual of $\{E_{mb}\mathcal{T}_{na}f\}_{m,n \in \mathbb{Z}}$?.

\section*{Conflict of interest}
The authors declare that they have no conflict of interest.

\section*{Data Availability Statement}

 The authors declare that no data has been used for this research.



\end{document}